\documentclass[11pt, notitlepage]{article}   

\usepackage[title]{appendix}

\usepackage{amsmath,amsthm,amsfonts,hyperref}   

\usepackage{bbm} 
\usepackage{amscd, array}
\usepackage{amsfonts}
\usepackage{amssymb}
\usepackage{color}      
\usepackage{epsfig}
\usepackage{graphicx}           
\usepackage{algorithm}
\usepackage[noend]{algpseudocode}

\usepackage{tikz}
\usetikzlibrary{shapes}
\usetikzlibrary{matrix}
\usetikzlibrary{positioning}
\usetikzlibrary{fit}

\tikzstyle{res}=[circle,thick,minimum size=4mm,draw=black,fill=red,inner sep=1pt]
\tikzstyle{non-res}=[circle,thick,minimum size=4mm,draw=black,inner sep=1pt]
\tikzstyle{light-res}=[circle,thick,minimum size=4mm,draw=black,fill=red!40,inner sep=1pt]
\tikzstyle{blue}=[circle,thick,minimum size=4mm,draw=black,fill=blue!20,inner sep=1pt]

\newtheorem{theorem}{Theorem}[section]

\def\diam{{\hbox{diam}}}

\newcommand{\edeg}{\textnormal{edeg}}

\newcommand{\cc}{\textnormal{cc}}
\newcommand{\bc}{\textnormal{bc}}
\newcommand{\ec}{\textnormal{ec}}
\newcommand{\ecc}{\textnormal{ecc}}
\newcommand{\eecc}{\textnormal{eecc}}
\newcommand{\mo}{\textnormal{Mo}}
\newcommand{\irr}{\textnormal{irr}}
\newcommand{\peri}{\textnormal{peri}}
\newcommand{\spr}{\textnormal{spr}}

\newcommand{\eperi}{\textnormal{eperi}}
\newcommand{\espr}{\textnormal{espr}}

\newtheorem{cor}[theorem]{Corollary}

\newtheorem{thm}[theorem]{Theorem}

\def\finf{\mathop{{\rm I}\kern -.27 em {\rm F}}\nolimits}

\usepackage[colorinlistoftodos,textsize=tiny]{todonotes}
\newcommand{\Comments}{1}
\newcommand{\mynote}[2]{\ifnum\Comments=1\textcolor{#1}{#2}\fi}
\newcommand{\mytodo}[2]{\ifnum\Comments=1%
  \todo[linecolor=#1!80!black,backgroundcolor=#1,bordercolor=#1!80!black]{#2}\fi}

\begin{document}


\title{Peripherality in networks: theory and applications}
\author{
	Jesse Geneson \and Shen-Fu Tsai}


\maketitle

\date{}

\begin{abstract}
We investigate several related measures of peripherality and centrality for vertices and edges in networks, including the Mostar index which was recently introduced as a measure of peripherality for both edges and networks. We refute a conjecture on the maximum possible Mostar index of bipartite graphs from (Do\v{s}li\'{c} et al, Journal of Mathematical Chemistry, 2018) and (Ali and Do\v{s}li\'{c}, Applied Mathematics and Computation, 2021). We also correct a result from the latter paper, where they claimed that the maximum possible value of the terminal Mostar index among all trees of order $n$ is $(n-1)(n-2)$. We show that this maximum is $(n-1)(n-3)$ for $n \ge 3$, and that it is only attained by the star.

We asymptotically answer another problem on the maximum difference between the Mostar index and the irregularity of trees from (F. Gao et al, On the difference of Mostar index and irregularity of graphs, Bulletin of the Malaysian Mathematical Sciences Society, 2021). We also prove a number of extremal bounds and computational complexity results about the Mostar index, irregularity, and measures of peripherality and centrality. 

We discuss graphs where the Mostar index is not an accurate measure of peripherality. We construct a general family of graphs with the property that the Mostar index is strictly greater for edges that are closer to the center. We also investigate centrality and peripherality in two graphs which represent the SuperFast and MOZART-4 systems of atmospheric chemical reactions by computing various measures of peripherality and centrality for the vertices and edges in these graphs. For both of these graphs, we find that the Mostar index is closer to a measure of centrality than peripherality of the edges. We also introduce some new indices which perform well as measures of peripherality on the SuperFast and MOZART-4 graphs. 
\end{abstract}




\section{Introduction}

Measures of centrality in networks are used as proxies for the importance or influence of nodes. In a social network, a user with highest degree centrality has the most connections in the network. In a network of websites, the website with the highest indegree centrality has the most links from other websites in the network. 

Centrality can also be used to understand systems of chemical reactions. For example, Silva et al. \cite{silva} represented systems of atmospheric chemical reactions as directed graphs. To understand which chemical species were the most important reactants within the structure of the chemical mechanisms in the system, they determined the outdegree centrality of the nodes in the directed graphs. There are various measures of centrality, several of which we discuss in the next subsection.

Peripherality is the opposite of centrality. While central vertices are the most important within the structure of a network, peripheral vertices are the least important. Given any measure of centrality, it can be inverted to produce a measure of peripherality. 

In addition to nodes, centrality and peripherality are also defined for edges in a network. We discuss versions of degree centrality and eccentricity for edges. The Mostar index was recently introduced as a measure of peripherality of the edges in a graph \cite{mostar1}. 

In this paper, we develop the theory of centrality and peripherality by answering open problems about the Mostar index and its variants, and proving many additional extremal and exact results about various measures of peripherality. We also construct some families of graphs where the Mostar index is not an accurate measure of peripherality. We apply various measures of centrality and peripherality to analyze the reactions in two systems of atmospheric chemical reactions called SuperFast and MOZART-4, and we find for both systems that the Mostar index is not an accurate measure of peripherality.

\subsection{Measures of centrality and peripherality among vertices}

We consider several centrality measures for vertices of a graph. In addition to ranking vertices by their centrality, these measures can also be used to rank vertices with respect to peripherality, since peripherality is the opposite of centrality. We use the standard definitions for all of these centrality measures.

The \emph{degree of vertex $v$ in $G$}, denoted by $\deg(v)$ when $G$ is clear from the context, is the number of edges in $G$ which contain $v$. This is used as a measure of centrality since often the most central vertices have the highest degree. The \emph{closeness centrality of $v$ in $G$}, denoted by $\cc(v)$, is the reciprocal of the sum of the distances from $v$ to each vertex in $G$. The most central vertices according to $\cc(v)$ are those that minimize the sum of the distances to the other vertices. The \emph{eccentricity of $v$ in $G$}, denoted by $\ecc(v)$, is the maximum possible distance from $v$ to any vertex in $G$. A vertex with minimum eccentricity in $G$ is called a \emph{center} of $G$.

The \emph{betweenness centrality of vertex $v$ in $G$}, denoted by $\bc(v)$, is the sum of the ratio of the number of shortest paths from $u$ to $w$ which pass through $v$ to the total number of shortest paths from $u$ to $w$, over all pairs of distinct vertices $u, w \in V(G)$ for which $u,v,w$ are all distinct. With respect to this measure, more of the shortest paths pass through central vertices than through peripheral vertices. The \emph{eigenvector centrality of $v$ in $G$}, denoted by $\ec(v)$, is the $v^{th}$ coordinate of the unit eigenvector with all coordinates non-negative, which corresponds to the maximum eigenvalue of the adjacency matrix of $G$. With respect to this measure, the most central vertices are those with the greatest coordinates in the eigenvector, and the most peripheral vertices are those with the least coordinates. We also define two alternative measures of centrality and peripherality in Section~\ref{vperi}.

\subsection{Measures of centrality and peripherality among edges}

As with vertices, it is natural to compare the centrality of edges in a network. The definitions of degree and eccentricity can both be easily adapted to edges. 

We define the \emph{edge degree of edge $e = \left\{u,v\right\}$ in $G$}, denoted by $\edeg(e)$, to be the number of vertices in $V(G)$ which are not equal to $u$ or $v$ and are adjacent to at least one of $u$ or $v$. As with vertex degree, this can be used as a measure of centrality since the most central edges often have the highest edge degree. Observe that if $e = \left\{u,v\right\}$, then $$\max(\deg(u),\deg(v))-1 \le \edeg(e) \le \deg(u)+\deg(v)-2.$$ 

We can construct edges in graphs that attain both the upper and lower bound from the last sentence. For the upper bound, let $G$ be the graph obtained from the disjoint union of $K_{1, m-1}$ and $K_{1, n-1}$ by adding an edge between the centers $c$, $d$ respectively of the stars. Then $\deg(c) = m$, $\deg(d) = n$, and $\edeg(\left\{c,d\right\}) = m+n-2$. For the lower bound, let $H$ be the graph obtained from $K_{1,m-1}$ by adding a new vertex $v$ with an edge between $v$ and the center vertex $c$ of $K_{1,m-1}$ and $n-1$ other edges between $v$ and $n-1$ leaf vertices of $K_{1,m-1}$, where $n \le m$. Then $\deg(c)= m$, $\deg(v) = n$, and $\edeg(\left\{c,v\right\}) = m-1$.

If $e = \left\{u,v\right\}$ is an edge in $G$ and $w$ is a vertex in $V(G)$, then the \emph{distance from $e$ to $w$}, denoted $d(e, w)$, is equal to $\min(d(u, w), d(v, w))$. The \emph{edge eccentricity of $e = \left\{u,v\right\}$ in $G$}, denoted by $\eecc(e)$, is the maximum possible value of $d(e, w)$ over all vertices $w \in V(G)$. As with eccentricity of vertices, note that centrality decreases as edge eccentricity increases. In other words, one way to define the most peripheral edges is those with the highest edge eccentricity. Observe that if $e = \left\{u,v\right\}$, then $$\min(\ecc(u),\ecc(v))-1 \le \eecc(e) \le \min(\ecc(u),\ecc(v)).$$ 

Again, we can construct edges in graphs that attain both the upper and lower bound from the last sentence. For the upper bound, let $e$ be an edge of $P_{2n+1}$ which is incident to the center vertex $c$, so $e = \left\{v,c\right\}$. Then $\ecc(c) = n$, $\ecc(v) = n+1$, and $\eecc(e) = n$. For the lower bound, let $e = \left\{c_1,c_2\right\}$ be the centermost edge of $P_{2n}$. Then $\ecc(c_1) = n$, $\ecc(c_2) = n$, and $\eecc(e) = n-1$.

Recently, a measure of peripherality for the edges of a network called the Mostar index was introduced in \cite{mostar1}, and it has already been investigated in dozens of papers. We discuss the Mostar index in more detail in the next subsection.

\subsection{Mostar index}
In \cite{mostar1}, Do\v{s}li\'{c} et al. introduced the \emph{Mostar index} of graphs as a measure of peripherality for edges and for networks. The Mostar index was also introduced independently in \cite{sharafdini}. The Mostar index is one of many topological indices for graphs, such as the Wiener index \cite{wienerindex}, the two Zagreb indices \cite{zagrebindex}, the Harary index \cite{hararyindex}, the Szeged index \cite{szegedindex}, degree-based indices \cite{degindex1, degindex2, degindex3}, and vertex and bond-additive indices \cite{addindex1, addindex2, addindex3}.

For any edge $\left\{u, v\right\} \in E(G)$, let $n_G(u, v)$ be the number of vertices in $G$ that are closer to $u$ than to $v$. The \emph{Mostar index of $e = \left\{u, v\right\}$} is defined as $\mo(e) = |n_G(u, v)-n_G(v, u)|$. The \emph{Mostar index of $G$} is defined as $\mo(G) = \sum_{\left\{u, v\right\} \in E(G)} |n_G(u, v)-n_G(v, u)|$. Note that $\mo(G) = \sum_{e \in E(G)} \mo(e)$.

Do\v{s}li\'{c} et al. proved that the maximum possible value of $\mo(G)$ over all graphs $G$ of order $n$ is $\Theta(n^3)$, and asked what is the exact maximum value of $\mo(G)$ over all graphs $G$ of order $n$. They also proved that the maximum possible value of $\mo(G)$ over all bipartite graphs $G$ of order $n$ is $\Theta(n^3)$, and they conjectured that the exact maximum value of $\mo(G)$ over all bipartite graphs $G$ of order $n$ is approximately $\frac{2n^3}{27}$. The same conjecture was stated in \cite{mostar2}. We refute their conjecture by exhibiting a family of bipartite graphs $G$ of order $n$ with $\mo(G) = \frac{n^3}{6\sqrt{3}}-O(n)$. 

The paper \cite{mostar2} also conjectured that the maximum possible value of $\mo(G)$ among all connected graphs $G$ of order $n$ is approximately $\frac{4}{27}n^3$. We sharpen the upper bound on the maximum possible value of $\mo(G)$ among all connected graphs $G$ of order $n$ to $\frac{5}{24}n^3(1+o(1))$.

Do\v{s}li\'{c} et al. also determined the minimum and maximum possible values of $\mo(T)$ over all trees $T$ of order $n$ in \cite{mostar1}, showing that the minimum value is $\lfloor \frac{(n-1)^2}{2} \rfloor$ and that this value is only attained by the path $P_n$ among all trees $T$ of order $n$. Moreover they showed that the maximum value is $(n-1)(n-2)$ and that this is attained only by the star $K_{1,n-1}$ of order $n$. In addition, they determined extremal bounds for unicyclic graphs of order $n$, and they computed the exact values of $\mo(G)$ for benzenoid graphs and Cartesian products. They asked what is the maximum possible Mostar index of chemical trees of order $n$, and this problem was solved in \cite{mostarindexchemtrees}. Do\v{s}li\'{c} et al. also asked in \cite{mostar1} which bicyclic graphs $G$ of order $n$ maximize the value of $\mo(G)$, and Tepeh answered this question in \cite{tepeh}.

Deng and Li \cite{degreeseq} investigated extremal bounds for the Mostar index of trees with a given order and degree sequence, as well as trees with a given order and independence number. Ghorbani et al. \cite{ghorbani} proved a lower bound of $2 \Delta (n-3)$ on the Mostar index of any tree of order $n \ge 3$ and maximum degree $\Delta$. Tratnik \cite{tratnik} generalized the definition of Mostar index to weighted graphs and showed that the Mozart indices of benzenoid systems can be computed in sub-linear time in the number of vertices.

\subsection{Mostar index versus irregularity}
The irregularity of a graph was introduced by Albertson \cite{albertson}. As with the Mostar index, irregularity is defined for both edges and graphs. For any edge $e = \left\{u,v\right\} \in E(G)$, the \emph{irregularity of $e$} is $\irr(e) = |d_u-d_v|$, where $d_v$ denotes the degree of vertex $v$. The \emph{irregularity of $G$} is defined as $\irr(G) = \sum_{\left\{u, v\right\} \in E(G)} |d_u-d_v|$. Note that $\irr(G) = \sum_{e \in E(G)} \irr(e)$.

Gao et al \cite{mo-irr} initiated the investigation of the difference $\mo(G)-\irr(G)$. They asked what is the maximum possible value of the difference $\mo(T)-\irr(T)$ among all trees $T$ of order $n$. We answer this question asymptotically by showing the maximum possible difference is $n^2(1-o(1))$. In fact we find two families of trees $T$ of order $n$ for which $\mo(T)-\irr(T) = n^2 - \Theta(\frac{\log{n}}{\log{\log{n}}}n)$.

One of the families that attain this bound are full $m$-ary trees of depth $d$ for a certain choice of $m$ and $d$. Another family are what we call \emph{factorial trees}, they were introduced in \cite{mo-irr} as a possible family of trees that attain the maximum of $\mo(T)-\irr(T)$ among all trees $T$ of order $n$. We determine the exact values of this difference for both families. We also determine the exact values for balanced spider graphs.

\subsection{Terminal Mostar index}
Ali and Do\v{s}li\'{c} \cite{mostar2} defined a variant of the Mostar index which is similar to a quantity that is present in the computation of the Colless index, a parameter which measures balance in phylogenetic trees \cite{collessterminal}. This variant only counts pendent vertices when measuring the contribution of each edge. 

For any edge $\left\{u, v\right\} \in E(G)$, let $\ell_G(u, v)$ be the number of pendent vertices in $G$ that are closer to $u$ than to $v$. The \emph{terminal Mostar index} is defined as $\mo^{\top}(G) = \sum_{\left\{u, v\right\} \in E(G)} |\ell_G(u, v)-\ell_G(v, u)|$. Note that the name is similar to the terminal Wiener index \cite{terminalwiener}, which sums the distances over all pairs of pendent vertices. However, the sum in the terminal Mostar index is over the edges of the graph, rather than pairs of pendent vertices as in the terminal Wiener index.

Ali and Do\v{s}li\'{c} claimed that the maximum possible value of the terminal Mostar index among all trees of order $n$ is $(n-1)(n-2)$. We show that this is incorrect. More specifically, we prove that this maximum is $(n-1)(n-3)$ for $n \ge 3$, and that it is only attained by the star.

We prove that the maximum possible terminal Mostar index among all bipartite graphs of order $n$ is $\frac{n^3}{27}(1 \pm o(1))$. We also prove that the maximum possible terminal Mostar index among all connected graphs of order $n$ is $\frac{n^3}{27}(1 + o(1))$.

\subsection{Total Mostar index}
Miklavi\v{c} and \v{S}parl \cite{totalmostar} introduced a variant of the Mostar index which sums $|n_G(u,v)-n_G(v,u)|$ over all subsets $\left\{u,v\right\} \subset V(G)$ rather than only over $\left\{u,v\right\} \in E(G)$. Specifically, we define $\mo^{*}(G) = \sum_{\left\{u,v\right\} \subset V(G)}|n_G(u,v)-n_G(v,u)|$. This parameter has been called both the \emph{total Mostar index} of $G$ and the \emph{distance-unbalancedness} of $G$. 

In \cite{totalmostar}, Miklavi\v{c} and \v{S}parl determined the value of $\mo^{*}(G)$ for paths, wheels, complete multipartite graphs, and other families of graphs. They posed the problems of determining the maximum and minimum possible values of $\mo^{*}(T)$ over all trees $T$ of order $n$ and to characterize the extremal trees.

Kramer and Rautenbach \cite{KrRa1} proved that stars are the unique trees of order $n$ with the minimum possible total Mostar index among all trees of order $n$. The same authors also proved in \cite{KrRa} that the maximum possible total Mostar index among all trees of order $n$ is $\frac{1}{2}n^3(1-o(1))$. They proved this by showing that the maximum possible total Mostar index among all spiders of order $n$ with $k$ legs is equal to $(\frac{1}{2}-\frac{5}{6k}+\frac{1}{3k^2})n^3+O(kn^2)$.

We completement the results of Kramer and Rautenbach by determining the exact value of $\mo^{*}(S)$ for every balanced spider $S$. This also gives a simpler proof of the fact that the maximum possible total Mostar index among all trees of order $n$ is $\frac{1}{2}n^3(1-o(1))$. We generalize this result by showing that the maximum possible total Mostar index among all graphs of order $n$ and degeneracy $k$ is $\frac{1}{2}n^3(1-o(1))$, for any $k = o(\sqrt{n})$.

\subsection{Alternative peripherality measures for vertices}\label{vperi}
The Mostar index and the total Mostar index are defined for the whole network, but they are not defined for single nodes. There is a natural way to define a version of the Mostar index for nodes, but it does not give an accurate measure of peripherality. Given a vertex $v \in V(G)$, we can define $\mo(v) = \frac{1}{2}\sum_{u \in V(G): \left\{u,v\right\} \in E(G)} |n_G(u,v)-n_G(v,u)|$, so that 
\begin{align*}
\sum_{v \in V(G)} \mo(v) =  \\
\frac{1}{2}\sum_{(u, v): \left\{u,v\right\} \in E(G)} |n_G(u,v)-n_G(v,u)| =  \\
\sum_{\left\{u, v\right\}: \left\{u,v\right\} \in E(G)} |n_G(u,v)-n_G(v,u)| = \\
\mo(G).
\end{align*}

Similarly, we can define $\mo^{*}(v) = \frac{1}{2}\sum_{u \in V(G): u \neq v} |n_G(u,v)-n_G(v,u)|$, so that 
\begin{align*}
\sum_{v \in V(G)} \mo^{*}(v) =  \\
\frac{1}{2}\sum_{(u, v): u \neq v} |n_G(u,v)-n_G(v,u)| =  \\
\sum_{\left\{u, v\right\}: u \neq v} |n_G(u,v)-n_G(v,u)| = \\
\mo^{*}(G).
\end{align*}

With these definitions, the measures $\mo(v)$ and $\mo^{*}(v)$ are closer to measures of centrality than peripherality for star graphs. If $c$ is the center vertex of $K_{1, n}$, then $\mo(c) = \mo^{*}(c) = \frac{1}{2}n(n-1)$ and $\mo(v) = \mo^{*}(v) = \frac{n-1}{2}$ for every leaf vertex $v \in V(K_{1, n})$. Intuitively the center vertex should be less peripheral than the leaf vertices of a star, but $\mo(c)$ and $\mo^{*}(c)$ are greater than $\mo(v)$ for all leaf vertices $v$. 

One of the issues with defining peripherality of a vertex this way is that $\left\{u,v\right\} \in E(G)$ contributes $|n_G(u,v)-n_G(v,u)|$ to $\mo(v)$ and $\mo^{*}(v)$, regardless of whether $n_G(u,v)<n_G(v,u)$ or $n_G(u,v)>n_G(v,u)$. If $n_G(v,u)$ is sufficiently greater than $n_G(u,v)$ for all $u \in V(G)-v$, as in the case when $v$ is the center vertex of a star, then $v$ will have the highest value of $\mo(v)$ among all of the vertices in $G$. In this case, $v$ intuitively should be less peripheral than the other vertices $u$ in $G$, but $\mo(v) > \mo(u)$ for all $u \in V(G)-v$. 

This led us to define an alternative definition for peripherality of vertices which still uses the function $n_G(u, v)$. Define $\peri(v)$ (the \emph{peripherality} of $v$) as the number of vertices $u$ in $G$ such that $n_G(u,v)>n_G(v,u)$, and $\peri(G)=\sum_{v\in V(G)}\peri(v)$. In words, the peripherality of a vertex $v$ is the number of other vertices that have more vertices closer to them than to $v$. The intuition in this definition is that if a vertex $v$ is closer to the periphery of a graph $G$, then there will be more vertices $u$ in $G$ for which $n_G(u,v) > n_G(v,u)$ (i.e. $u$ is closer to more of the vertices in $G$ than $v$ is). 

For the star graph $K_{1, n}$ with $n \ge 2$, we have $\peri(c) = 0$ and $\peri(v) = 1$ for all leaf vertices $v$. Thus $\peri(K_{1,n}) = n$, and the leaf vertices are more peripheral than the center vertex by this definition.

We also define another measure of peripherality for vertices called \emph{sum peripherality}. Define $\spr(v)=\sum_{u\in V(G)-v}n_G(u,v)$ and $\spr(G)=\sum_{v\in V(G)}\spr(v)$. Note that they can also be written as
$$
\peri(G)=\sum_{\{u,v\}\subset V(G)}\mathbbm{1}{\left[n_G(u,v)\neq n_G(v,u)\right]}
$$
$$
\spr(G)=\sum_{\{u,v\}\subset V(G)}n_G(u,v)+n_G(v,u)
$$

Observe that for the star graph $K_{1, n}$ with $n \ge 2$, we have $\spr(c) = n$ and $\spr(v) = 2n-1$ for all leaf vertices $v$. Thus $\spr(K_{1,n}) = 2n^2$, and the leaf vertices are more peripheral than the center vertex by this definition as well.

\subsection{Alternative peripherality measures for edges}\label{edgeperi}

We also define two alternative measures of peripherality for edges. These are both analogous to the measures $\peri(v)$ and $\spr(v)$ for vertices.

Define $\eperi(e)$ (the \emph{edge peripherality} of $e = \left\{u, v\right\}$) as the number of vertices $x$ in $G$ which are not contained in $e$ such that $n_G(x,u)>n_G(u,x)$ and $n_G(x, v) > n_G(v,x)$, and $\eperi(G)=\sum_{e\in E(G)}\eperi(e)$. In words, the edge peripherality of an edge $e$ is the number of vertices not in $e$ that have more vertices closer to them than to both vertices in $e$. The intuition in this definition is that if an edge $e = \left\{u, v\right\}$ is closer to the periphery of a graph $G$, then there will be more vertices $x$ in $G$ which are not contained in $e$ such that $n_G(x,u)>n_G(u,x)$ and $n_G(x, v) > n_G(v,x)$ (i.e. $x$ is closer to more of the vertices in $G$ than either of the vertices in $e$ is). 

For the star graph $K_{1, n}$ with $n \ge 2$, we have $\eperi(e) = 0$ for all edges $e$ in the graph. Thus $\eperi(K_{1,n}) = 0$.

We also define another measure of peripherality for edges called \emph{edge sum peripherality} that is analogous to $\spr(v)$. Given edge $e = \left\{u,v\right\}$, define $\espr(e)=\sum_{x \in V(G)-\left\{u,v\right\}}(n_G(x,u)+n_G(x,v))$ and $\espr(G)=\sum_{e\in E(G)}\espr(e)$. 

Observe that for the star graph $K_{1, n}$ with $n \ge 2$, we have $\espr(e) = 2n-2$ for all edges $e$ in the graph. Thus $\espr(K_{1,n}) = 2n^2 - 2n$.

\subsection{Representing atmospheric chemical reactions with graphs}

Multiple groups have used graphs to model and analyze atmospheric phenomena. Chaudhuri and Middey \cite{chmi} used a bipartite graph model to forecast thunderstorms over Kolkata. For the vertex set of the bipartite graph in their model, one part consisted of two time vertices and the other part consisted of the four meteorological parameters of temperature, relative humidity, convective available potential energy, and convective inhibition energy.

Silva et al. \cite{silva} used a directed graph model to analyze three systems of atmospheric chemical reactions called SuperFast, GEOS-Chem v12.6, and the Master Chemical Mechanism v3.3. SuperFast is the smallest system of the three with fewer than $20$ chemical species, while the GEOS-Chem model has around $200$ species and the Master Chemical Mechanism has around $6000$ species. SuperFast is intended to be used in global climate simulations with a duration of tens of years or greater, while GEOS-Chem is intended for when the duration ranges from a few days to a few years.

In the directed graph model of \cite{silva}, the vertex set consists of chemical species and reactions. For each reaction $R$, the reactants in $R$ have edges to $R$, and the products in $R$ have edges from $R$. Silva et al. examined the outdegree centrality of the vertices in the directed graphs corresponding to each system of reactions to determine the most important reactants within the structure of the network of atmospheric chemical reactions. The chemical species $O H$ had the highest outdegree centrality in all three directed graphs, with $H O_2$ second in SuperFast and the Master Chemical Mechanism, but second to $N O$ in GEOS-Chem.

We define an undirected graph on the reactants of SuperFast and analyze various measures of centrality and peripherality for the vertices and edges to determine the most and least important reactants and reactions within the structure of the system. Our analysis of the reactions in SuperFast includes computing the Mostar index, the edge peripherality, and the edge sum peripherality of all edges in the corresponding graph, as well as the edge degree and edge eccentricity. We also perform a similar analysis on MOZART-4, another system of atmospheric chemical reactions. Our code can be found at \cite{gtm4code}.

For both of these graphs, we found that the Mostar index is not an accurate measure of peripherality. More specifically, the more central edges in these graphs generally have higher Mostar indices than the more peripheral edges. We also construct a general family of graphs for which the Mostar index is strictly greater for edges that are closer to the center. This is the opposite of path graphs and balanced spider graphs, where the Mostar index increases with distance from the center.

\subsection{Outline of paper}

In Section \ref{refute}, we refute the conjecture from \cite{mostar1} about the maximum possible Mostar index for bipartite graphs by showing that there exist complete bipartite graphs of order $n$ with Mostar index at least $\frac{n^3}{6\sqrt{3}}-6n$. In Section \ref{terminalmostar}, we show that the maximum possible terminal Mostar index among all bipartite graphs of order $n$ is $\frac{n^3}{27}(1 \pm o(1))$. We also show the stronger result that the maximum possible terminal Mostar index among all connected graphs of order $n$ is $\frac{n^3}{27}(1 \pm o(1))$. We also correct the result from \cite{mostar2} on the maximum possible terminal Mostar index for trees.

In Section \ref{mostarirreg}, we find two different families of trees $T$ of order $n$ which attain a bound of $\mo(T) - \irr(T) = n^2 - \Theta(\frac{\log{n}}{\log{\log{n}}}n)$, asymptotically answering the question from \cite{mo-irr}. In Section \ref{totalmostar}, we prove that the maximum possible value of $\mo^{*}(G)$ among all graphs $G$ of order $n$ and degeneracy $k$ with $k = o(\sqrt{n})$ is $\frac{1}{2}n^3(1-o(1))$. We also determine the exact value of the total Mostar index for every balanced spider graph.

In Section \ref{sprres}, we determine the exact values of $\spr(G)$ for paths, cycles, and balanced spiders. We also prove an upper bound on $\spr(G)$ in terms of the order and diameter. In Section \ref{esprres}, we determine the minimum possible value of $\espr(G)$ among all connected graphs $G$ of order $n$, and we determine the edge sum peripherality of complete graphs and complete bipartite graphs. 

In Section \ref{perires}, we determine the exact values of $\peri(G)$ for vertex-transitive graphs, paths, and balanced spiders. We also determine the exact maximum possible value of $\peri(G)$ among all graphs $G$ of order $n$ for each $n \ge 9$. Moreover we determine the maximum possible value of $\peri(v)+\deg(v)$ in any graph of order $n$, as well as any tree of order $n$. 

In Section \ref{eperires}, we determine the maximum possible value of $\eperi(e)+\edeg(e)$ in any graph of order $n$, as well as any tree of order $n$. We also asymptotically determine the maximum possible value of $\eperi(T)$ among all trees $T$ of order $n$. Moreover we determine the exact values of the edge peripherality of cycles, complete graphs, complete bipartite graphs, paths, and balanced spiders.

In Section \ref{npcliqueproblems}, we prove NP-completeness of several problems about Mostar index, irregularity, vertex peripherality, edge peripherality, vertex eccentricity, and edge eccentricity in cliques. In Section \ref{mostarnot}, we construct some families of graphs where the Mostar index is not an accurate measure of peripherality. Finally in Section \ref{superfastperi} and Section \ref{mozart4peri}, we analyze centrality and peripherality in SuperFast and MOZART-4.

Several results in this paper involve \emph{spider graphs}, which are trees with a single vertex of degree at least $3$. The \emph{legs} of a spider graph are the disjoint subgraphs obtained from removing the single vertex of degree at least $3$, and the length of a leg is the number of vertices on the leg. A spider graph is called \emph{balanced} if all of its legs have the same length. Note that a balanced spider with $2$ legs is a path of odd order, and a spider with legs of length $1$ is a star.

\section{On the Mostar index and the terminal Mostar index}

In this section, we prove extremal results about the Mostar index and the terminal Mostar index. In the first subsection, we show that there exist connected bipartite graphs of order $n$ with Mostar index at least $\frac{n^3}{6\sqrt{3}}-6n$, refuting the conjecture from \cite{mostar1}. In the second subsection, we asymptotically determine the maximum possible terminal Mostar index among all bipartite graphs of order $n$ and the maximum possible terminal Mostar index among all connected graphs of order $n$. We also correct the result from \cite{mostar2}.

\subsection{Maximum Mostar index on bipartite graphs}\label{refute}

In \cite{mostar1}, the authors conjectured that the maximum possible Mostar index among all connected bipartite graphs of order $n$ is approximately $\frac{2}{27}n^3$. This same conjecture was also stated as Conjecture 6.1 in \cite{mostar2}. We refute the conjecture by showing that there exist connected bipartite graphs of order $n$ with Mostar index at least $\frac{n^3}{6\sqrt{3}}-6n$. 

\begin{thm}
The maximum possible Mostar index among all complete bipartite graphs of order $n$ is $\frac{n^3}{6\sqrt{3}}-O(n)$.
\end{thm}

\begin{proof}
Consider the complete bipartite graph $G = K_{x, n-x}$ with $x \le n-x$. This has Mostar index $x(n-x)(n-2x)$, since there are $x(n-x)$ edges and $|n_G(u, v)-n_G(v, u)| = n-2x$ for every edge $\left\{u, v\right\} \in E(G)$.

Let $f(x) = x(n-x)(n-2x)$. Thus $f'(x) = n^2-6nx+6x^2$, which has roots $r_1 = n(\frac{1}{2}-\frac{1}{2\sqrt{3}})$ and $r_2 = n(\frac{1}{2}+\frac{1}{2\sqrt{3}})$. Note that the second root $r_2$ is out of range since $x \le n-x$, but the first root $r_1$ is a local maximum since $f''(x) = -6(n-2x)$ and $r_1 < \frac{n}{2}$. Again using the fact that $f''(x) = -6(n-2x)$, we have $|f'(x)| \le 6n$ for any $x \in (r_1-1,r_1+1)$. 

Thus there exists an integer $t \in (r_1-1,r_1+1)$ for which $\mo(K_{t, n-t}) \ge f(r_1)-6n = \frac{n^3}{6\sqrt{3}}-6n$. Moreover we have $\mo(K_{x, n-x}) \le f(r_1) = \frac{n^3}{6\sqrt{3}}$ for all $x \in [0, \frac{n}{2}]$ since $f(x)$ has a local maximum at $x = r_1$, $f(x)$ increases on the interval $[0, r_1]$ and $f(x)$ decreases on the interval $[r_1, \frac{n}{2}]$. 
\end{proof}

The bound in the last result is nearly attained by $K_{\frac{n}{5},\frac{4n}{5}}$, which has Mostar index $\frac{12n^3}{125}$ when $n$ is divisible by $5$.

\begin{cor}
There exist connected bipartite graphs of order $n$ with Mostar index at least $\frac{n^3}{6\sqrt{3}}-6n$. 
\end{cor}

It is still an open problem to determine the maximum possible Mostar index among all connected graphs of order $n$. It is clearly $\Theta(n^3)$ since graphs of order $n$ have at most $\binom{n}{2}$ edges and $|n_G(u, v)-n_G(v, u)| \le n$ for all $\left\{u, v\right\} \in E(G)$. This upper bound of $\frac{1}{2}n^3$ was observed in \cite{mostar1}, where they also constructed a family of connected graphs $G$ of order $n$ with $\mo(G) \approx \frac{4n^3}{27}$. In the next result, we show a simple argument which gives an improved upper bound for the maximum possible Mostar index among all connected graphs of order $n$.

\begin{thm}
The maximum possible Mostar index among all connected graphs of order $n$ is at most $\frac{5}{24}n^3(1 + o(1))$.
\end{thm}

\begin{proof}
Let $G$ be a connected graph of order $n$ with vertices $v_1, v_2, \dots, v_n$ of degree $d_1 \ge d_2 \ge \dots \ge d_n$ respectively. For each vertex $v_i$, let $e_i$ be the number of vertices $v_j$ with $j < i$ for which $\left\{v_i,v_j\right\} \in E(G)$.

Among all of the edges $\left\{v_i,v_j\right\} \in E(G)$ with $j < i$, observe that $|n_G(v_i,v_j)-n_G(v_j,v_i)| < n-d_i \le n - e_i$. Thus the contribution of the edges $\left\{v_i,v_j\right\} \in E(G)$ with $j < i$ to $\mo(G)$ is at most $e_i(n-e_i)$. Also note that $e_i < i$ for each $i = 1, 2, \dots, n$.

Thus $\mo(G) \le \sum_{i = 1}^n e_i(n-e_i)$. For all $i > \frac{n}{2}$, we have $e_i (n-e_i) \le \frac{n^2}{4}$. For $i \le \frac{n}{2}$, we have $e_i(n-e_i) \le i(n-i)$. Thus 
\begin{align*}
\mo(G) \le  \\
\left(\frac{n}{2}\right) \left(\frac{n^2}{4}\right)+\sum_{i = 1}^{\lfloor \frac{n}{2} \rfloor} i(n-i) \le  \\
\frac{n^3}{8}+n \left( \frac{\frac{n}{2}(\frac{n}{2}+1)}{2} \right)-\frac{\frac{n-1}{2}(\frac{n-1}{2}+1)n}{6} = \\
\frac{5}{24}n^3+O(n^2).
\end{align*}
\end{proof}

\subsection{Extremal results for terminal Mostar index}\label{terminalmostar}

In this subsection, we prove a sharp asymptotic bound on the maximum of $\mo^{\top}(G)$ among all bipartite graphs $G$ of order $n$ and among all connected graphs $G$ of order $n$. The minimum is clearly $0$, since any even cycle is bipartite and has no leaves.

\begin{thm}\label{maxterm}
The maximum possible terminal Mostar index among all bipartite graphs of order $n$ is $\frac{n^3}{27}(1 \pm o(1))$.
\end{thm}

\begin{proof}
For the upper bound, let $G$ be any bipartite graph of order $n$. Suppose that $G$ has $q$ leaves, so there are $n-q$ non-leaves in $G$. Let $a$ be the number of non-leaves in the left part of $G$ and let $b$ be the number of non-leaves in the right part of $G$, so $a+b = n-q$. There are a total of at most $q+a b$ edges in $G$. The $q$ edges are the edges adjacent to the leaves, and there are at most $a b$ edges between non-leaves. 

Each of the $q$ edges contributes at most $q$ to $\mo^{\top}(G)$, since there are $q$ leaves. Thus their total contribution is at most $q^2$. Each of the $a b$ edges also contribute at most $q$ to $\mo^{\top}(G)$, so their total contribution is at most $a b q$. Since $a+b+q = n$, by the arithmetic mean geometric mean inequality we obtain 
\begin{align*}
\mo^{\top}(G) \le  \\
q^2 + a b q \le  \\
q^2 + \left( \frac{n}{3} \right)^3 = \\
\frac{n^3}{27}+O(n^2).
\end{align*}

For the lower bound, we can construct a bipartite graph with $a$, $b$, and $q$ as close to equal as possible, and all of the leaves adjacent to the same vertex, giving a graph $G$ with $\mo^{\top}(G) \ge \frac{n^3}{27}-O(n^2)$.
\end{proof}






We define a {\it leg} as an induced subgraph of a maximal connected sequence of vertices with degree two together with a connected leaf. An edge is a {\it leg edge} if it belongs to or connects a leg. The number of legs is equal to the number of leaves and less than or equal to the number of leg edges. The number of leg edges is equal to the number of leg vertices.
\begin{thm}
For every $n$, there is a connected graph of order $n$ with maximum terminal Mostar index such that all its legs are connected to the same vertex.
\end{thm}
\begin{proof}
Consider a connected graph $G$ with maximum terminal Mostar index and not all its legs are adjacent to the same vertex. Suppose it has $K$ non-leg vertices $v_1,\ldots,v_K$ and $n-K$ leg vertices, and let $x_k$ be the number of legs connected to $v_k$ for $k=1,2,\ldots,K$. The number of leaves $L=x_1+\ldots+x_K$. Suppose there exists $i\ne j$ such that $x_i$ and $x_j$ are both non-zero. Given the values of $x_1,\ldots,x_K$ we can express the terminal Mostar index as $(n-K)(L-2)$ contributed by leg edges plus the sum of $|E(G)|-(n-K)$ terms where each term has no absolute-value bar and the combination of $x_i$ and $x_j$ is either $x_i+x_j$, $x_i-x_j$, $-x_i+x_j$, $-x_i-x_j$, $x_i$, $-x_i$, $x_j$, $-x_j$, or none. Let there be $a$ non-zero terms with $x_i-x_j$, $b$ zero terms permitting both $x_i-x_j$ and $-x_i+x_j$, $c$ non-zero terms with $x_j-x_i$. $d$ non-zero terms with $x_i$ or $-x_j$, $e$ zero terms permitting either both $x_i$ and $-x_i$ or both $x_j$ and $-x_j$, and $f$ non-zero terms with $-x_i$ or $x_j$. Optimality implies
\begin{align*}
2b+2c+e+f&\leq 2a+d\\
2b+2a+e+d&\leq 2c+f.
\end{align*}
The first and second inequalities above correspond to moving one leg vertex from $v_i$ to $v_j$ and $v_j$ to $v_i$, respectively. We then have $b=0$ and $e=0$, and
$$
2c+f=2a+d.
$$
So moving one leg from $v_i$ to $v_j$ still achieves optimal terminal Mostar index. 
We can keep moving legs until all legs are connected to the same vertex.
\end{proof}

Using the last result, we determine the maximum possible terminal Mostar index among all connected graphs of order $n$.

\begin{thm}
The maximum possible terminal Mostar index among all connected graphs of order $n$ is $\frac{n^3}{27}(1 \pm o(1))$.
\end{thm}
\begin{proof}
Suppose the connected graph of order $n$ with maximum possible terminal Mostar index has $n_0$ leaves, $n_1\ge n_0$ leg vertices, and all legs are connected to the same vertex $v$. Let $G'$ be a subgraph of $G$ obtained by removing these $n_0$ legs. For $k=1,2,\ldots,K$, let $n_{k+1}$ be the number of vertices in $G'$ that are at distance $k$ from $v$. Then to maximize the terminal Mostar index, for each $k=1,2,\ldots,K-1$, $(u,w)\in E(G)$ if $u$ and $w$ are at distance $k$ and $k+1$ from $v$, respectively. The terminal Mostar index of $G$ is
\begin{align*}
\mo^{\top}(G)&= n_1(n_0-2)+n_0(n_2+n_2n_3+n_3n_4+\ldots+n_Kn_{K+1}).\\
\end{align*}
Note that $1+n_1+\ldots+n_{K+1}=n$. Apparently $n_0=n_1$ i.e. all legs have length one, and
\begin{align*}
\mo^{\top}(G)&= n_1(n_1-2)+n_1(n_2+n_2n_3+n_3n_4+\ldots+n_Kn_{K+1}).\\
&=n_1(n_2n_3+n_3n_4+\ldots+n_Kn_{K+1})+O(n^2).
\end{align*}
Without loss of generality assume $n_K\ge n_3$, then without losing optimality let $n_2=1$ and
\begin{align*}
\mo^{\top}(G)&=  n_1(n_3n_4+\ldots+n_Kn_{K+1})+O(n^2).
\end{align*}
Again without loss of generality assume $n_K\ge n_4$, then without losing optimality let $n_3=1$ and 
\begin{align*}
\mo^{\top}(G)&=  n_1(n_4n_5+\ldots+n_Kn_{K+1})+O(n^2).
\end{align*}
This goes on until we have 
\begin{align*}
\mo^{\top}(G)&=n_1(n_{K-1}n_K+n_Kn_{K+1})+O(n^2)\\
&=n_1\left(n_K(n_{K-1}+n_{K+1})\right)+O(n^2) \text{ if $K>2$,}\\
\mo^{\top}(G)&=n_1n_2n_3+O(n^2) \text{ if $K=2$,}\\
\mo^{\top}(G)&=n_1n_2+O(n^2) \text{ if $K=1$.}\\
\end{align*}

\noindent In each line, the $O(n^2)$ term is less than $n_1(n_1+n_2+n_3+\dots n_{K-1}) < n^2$ since $1+n_1+\ldots+n_{K+1}=n$. When $K = 1$, the maximum is $O(n^2)$. When $K \ge 2$, the maximum is at most $\frac{n^3}{27}(1 + o(1))$ by the arithmetic mean geometric mean inequality. By Theorem~\ref{maxterm}, the maximum value is $\frac{n^3}{27}(1 \pm o(1))$.
\end{proof}

\subsection{Terminal Mostar index on trees}

In the next theorem, we determine the exact value of the terminal Mostar index of any spider graph, even if it is unbalanced. This contradicts a result in \cite{mostar2} where they claimed that $\mo^{\top}(K_{1,n-1}) = (n-1)(n-2).$

\begin{thm}
If $S$ is a spider of order $n$ with $k \ge 2$ legs, then $\mo^{\top}(S) = (k-2)(n-1)$.
\end{thm}

\begin{proof}
The spider $S$ has $n-1$ edges since it is a tree of order $n$. For each edge $\left\{u,v\right\} \in E(S)$ on leg $\ell$ with $v$ closer to the endpoint of $\ell$ than $u$ is, the only leaf that is closer to $v$ than to $u$ is the leaf at the end of leg $\ell$. All of the $k-1$ other leaves are closer to $u$. Thus each edge in $S$ contributes $(k-1)-1 = k-2$ to $\mo^{\top}(S)$, so we obtain $\mo^{\top}(S) = (k-2)(n-1)$.
\end{proof}

\begin{cor}
For all $n \ge 3$, we have $\mo^{\top}(K_{1,n-1}) = (n-1)(n-3).$
\end{cor}

Next, we correct a result from \cite{mostar2}, where they claimed that the maximum possible terminal Mostar index among all trees of order $n$ is $(n-1)(n-2)$. We show that the maximum is $(n-1)(n-3)$ for $n \ge 3$, and it is only attained by the star. 

\begin{thm}
Among all trees $T$ of order $n \ge 3$, the maximum possible value of $\mo^{\top}(T)$ is $(n-1)(n-3)$. The only tree $T$ of order $n$ which attains the maximum possible value is $T = K_{1,n-1}$.
\end{thm}

\begin{proof}
Any tree of order $n$ has $n-1$ edges. If $n \ge 3$, there must be at least one vertex in the tree that is not a leaf. Thus there are at most $n-1$ leaf vertices in the tree. 

For each edge $\left\{u,v\right\}$ in any tree $T$ of order $n$, there is at least one leaf in $T$ that is closer to $u$ and at least one leaf in $T$ that is closer to $v$. Since there are at most $n-1$ leaf vertices in $T$, the edge $\left\{u,v\right\}$ contributes at most $(n-2)-1 = n-3$ to $\mo^{\top}(T)$. Thus the maximum possible value of $\mo^{\top}(T)$ among all trees $T$ of order $n \ge 3$ is $(n-1)(n-3)$.

Note that the value $(n-1)(n-3)$ can only be attained if there are $n-1$ leaves in $T$, or else none of the edges could contribute $n-3$ to $\mo^{\top}(T)$. Thus the only tree $T$ of order $n$ which attains the maximum possible value of $\mo^{\top}(T)$ is $T = K_{1,n-1}$.
\end{proof}

If $n = 2$, the maximum value of $\mo^{\top}(T)$ among all trees $T$ of order $n$ from \cite{mostar2} is correct, since every vertex is a leaf when $n = 2$.

\section{Comparing the Mostar index and irregularity on trees}\label{mostarirreg}

Gao et al \cite{mo-irr} investigated the maximum possible value of the difference $\mo(T)-\irr(T)$ among all trees $T$ of order $n$. Interestingly, they found that the answer is exactly $n^2-7n+18$ for $7 \le n \le 22$. However, the pattern breaks at $n = 22$, where the maximum difference is $346$ instead of $348$. In this section, we prove that the maximum possible value of the difference $\mo(T)-\irr(T)$ among all trees $T$ of order $n$ is $n^2(1-o(1))$. More specifically, we find two different families of trees $T$ of order $n$ which both attain a bound of $\mo(T) - \irr(T) = n^2 - \Theta(\frac{\log{n}}{\log{\log{n}}}n)$.

\subsection{Factorial trees}

In \cite{mo-irr}, Gao et al introduced a family of trees which attained the maximum possible difference $\mo(T)-\irr(T)$ among all trees $T$ of order $n$ for $n = 3$, $7$, and $21$. We determine the value of $\mo(T)-\irr(T)$ for every tree $T$ in this family, and then we use this result to prove that the maximum possible value of the difference $\mo(T)-\irr(T)$ among all trees $T$ of order $n$ is $n^2(1-o(1))$.

Define $T_{!, m}$ as the rooted tree of depth $m-1$ such that the vertices at depth $i$ have degree $m-i$ for each $i = 0, \dots, m-1$. In the following proof, we use the convention that $\prod_{i = a}^{a-1} x_i = 1$ for any integer $a$ and sequence $x_i$.

\begin{thm}\label{t!m}
For all $m > 1$, we have 
\begin{enumerate}
\item $|E(T_{!, m})| = |V(T_{!, m})|-1 = m\sum_{i = 0}^{m-2} \prod_{j = 1}^{i} (m-1-j)$,
\item $\irr(T_{!, m}) = m\sum_{i = 0}^{m-2} \prod_{j = 1}^{i} (m-1-j)$, and
\item $\mo(T_{!,m}) = |E(T_{!, m})||V(T_{!, m})|-2 m \sum_{k = 1}^{m-1} \sum_{i = k-1}^{m-2} \prod_{j = 1}^{i} (m-1-j)$.
\end{enumerate}
\end{thm}

\begin{proof}
First note that $|E(T_{!, m})| = |V(T_{!, m})|-1$ since $T_{!,m}$ is a tree. We have $|V(T_{!, m})| = 1+m\sum_{i = 0}^{m-2} \prod_{j = 1}^{i} (m-1-j)$ since there are $m \prod_{j = 1}^{i} (m-1-j)$ vertices at depth $i+1$ for each $0 \le i \le m-2$.

We have $\irr(T_{!, m}) = m \sum_{i = 0}^{m-2} \prod_{j = 1}^{i} (m-1-j)$ since $|d_u-d_v| = 1$ for all edges $\left\{u, v \right\} \in E(T_{!, m})$.

For each edge $\left\{u,v\right\}$ between a vertex at depth $k-1$ and a vertex at depth $k$, we have $|n_{T_{!,m}}(u,v)-n_{T_{!,m}}(v,u)| = |V(T_{!, m})|-2 \sum_{i = k-1}^{m-2} \prod_{j = k}^{i} (m-1-j)$. The number of edges in $T_{!,m}$ between a vertex at depth $k-1$ and a vertex at depth $k$ is equal to $m\prod_{h = 1}^{k-1} (m-1-h)$. Thus we have 
\begin{align*}
\mo(T_{!,m}) =  \\
|E(T_{!, m})||V(T_{!, m})| - 2 m \sum_{k = 1}^{m-1} \prod_{h = 1}^{k-1} (m-1-h) \sum_{i = k-1}^{m-2} \prod_{j = k}^i (m-1-j) = \\
|E(T_{!, m})||V(T_{!, m})| - 2 m  \sum_{k = 1}^{m-1} \sum_{i = k-1}^{m-2} \prod_{j = 1}^{i} (m-1-j)
\end{align*}\end{proof}

\begin{cor}
There exists an infinite family of trees $T_n$ such that $T_n$ has order $n$ and $\mo(T) - \irr(T) = n^2 - \Theta(\frac{\log{n}}{\log{\log{n}}}n)$.
\end{cor}

\begin{proof}
Let $T_n = T_{!,m}$, so $n = |V(T_{!, m})| = 1 + m\sum_{i = 0}^{m-2} \prod_{j = 1}^{i} (m-1-j)$. First note that $\irr(T_{!,m}) = |V(T_{!,m})|-1$, so it suffices to prove that $\mo(T_{!,m}) = n^2 - \Theta(\frac{\log{n}}{\log{\log{n}}}n)$.

In Theorem~\ref{t!m}, we proved that $\mo(T_{!,m}) = |E(T_{!, m})||V(T_{!, m})|-2 m \sum_{k = 1}^{m-1} \sum_{i = k-1}^{m-2} \prod_{j = 1}^{i} (m-1-j)$. Note that $|E(T_{!, m})||V(T_{!, m})| = n^2-n$, so it suffices to prove that $m \sum_{k = 1}^{m-1} \sum_{i = k-1}^{m-2} \prod_{j = 1}^{i} (m-1-j) = \Theta(\frac{\log{n}}{\log{\log{n}}}n)$.

First we claim that $n = \Theta((m-1)!)$. The lower bound $n = \Omega((m-1)!)$ follows since 
\begin{align*}
n = 1 + m\sum_{i = 0}^{m-2} \prod_{j = 1}^{i} (m-1-j) \ge  \\
m \prod_{j = 1}^{m-2} (m-1-j) = \\
m(m-2)! = \Omega((m-1)!).
\end{align*}

To see that $n = O((m-1)!)$, note that
\begin{align*}
n-1 = m\sum_{i = 0}^{m-2} \prod_{j = 1}^{i} (m-1-j) =  \\
m (m-2)! \sum_{i = 0}^{m-2} \frac{1}{i!} < \\
e m (m-2)! = O((m-1)!).
\end{align*}

Thus we have $n = \Theta((m-1)!)$, so $m = \Theta(\frac{\log{n}}{\log{\log{n}}})$.

Now we prove the upper bound $m \sum_{k = 1}^{m-1} \sum_{i = k-1}^{m-2} \prod_{j = 1}^{i} (m-1-j) = O(\frac{\log{n}}{\log{\log{n}}}n)$. Note that we have
\begin{align*}
m \sum_{k = 1}^{m-1} \sum_{i = k-1}^{m-2} \prod_{j = 1}^{i} (m-1-j) \le  \\
m \sum_{k = 1}^{m-1} \sum_{i = 0}^{m-2} \prod_{j = 1}^{i} (m-1-j) \le \\
m \sum_{k = 1}^{m-1} e (m-2)! = e m! = \\
O(m n) = O(\frac{\log{n}}{\log{\log{n}}}n).
\end{align*}

Next we prove the lower bound $m \sum_{k = 1}^{m-1} \sum_{i = k-1}^{m-2} \prod_{j = 1}^{i} (m-1-j) = \Omega(\frac{\log{n}}{\log{\log{n}}}n)$. Note that we have
\begin{align*}
m \sum_{k = 1}^{m-1} \sum_{i = k-1}^{m-2} \prod_{j = 1}^{i} (m-1-j) \ge \\
m  (m-1) (m-2)! = m! = \Omega(m n) =  \Omega(\frac{\log{n}}{\log{\log{n}}}n).
\end{align*}
\end{proof}

\subsection{Full $m$-ary trees}

Let $F_{m, d}$ denote the full $m$-ary tree of depth $d$, which has $|V(F_{m, d})| = \frac{m^{d+1}-1}{m-1}$ and $|E(F_{m, d})| = -1+\frac{m^{d+1}-1}{m-1}$. We determine the exact value of $\irr(F_{m, d})$ and $\mo(F_{m, d})$ for all $m, d \ge 2$.

For $d = 1$, note that $F_{m, d} = K_{1, m}$, in which case we have $\irr(F_{m, 1}) = \mo(F_{m,1}) = m(m-1)$ \cite{albertson,mostar1,mo-irr}.

\begin{thm}
For all $m, d \ge 2$, we have
\begin{enumerate}
\item $\irr(F_{m, d}) = m+m^{d+1}$
\item $\mo(F_{m, d}) = |V(F_{m,d})|^2 - |V(F_{m,d})| - 2d(|V(F_{m,d})|+\frac{1}{m-1})+\frac{2}{m-1}(|V(F_{m,d})|-1)$
\end{enumerate}
\end{thm}

\begin{proof}
To see that $\irr(F_{m, d}) = m+m^{d+1}$, note that $d_v = m+1$ for all vertices in $F_{m, d}$ except the root and the leaves, so $|d_u-d_v| = 0$ for all edges $\left\{u, v\right\} \in E(F_{m, d})$ except for edges that contain the root and edges that contain leaves. If $\left\{u, v\right\}$ contains the root, then $|d_u-d_v| = 1$. There are $m$ edges that contain the root, so these edges contribute $m$ to the sum. Moreover, any edge $\left\{u, v\right\} \in E(F_{m, d})$ that contains a leaf satisfies $|d_u-d_v| = m$, and there are $m^d$ edges that contain leaves, so these edges contribute $m^{d+1}$ to the sum. Thus $\irr(F_{m, d}) = m+m^{d+1}$.

To determine $\mo(F_{m, d})$, we compute $|n_{F_{m,d}}(u, v)-n_{F_{m,d}}(v, u)|$ for each edge in $E(F_{m,d})$. If $\left\{u,v\right\}$ is between vertices $u$ and $v$ that are distances $i$ and $i-1$ from the root respectively, then $n_{F_{m,d}}(u, v) = \sum_{j = 0}^{d-i} m^j$ and $n_{F_{m,d}}(v, u) = |V(F_{m,d})|-\sum_{j = 0}^{d-i} m^j$. Thus

\begin{align*}
|n_{F_{m,d}}(u, v)-n_{F_{m,d}}(v, u)| = \\
|V(F_{m,d})|-2\sum_{j = 0}^{d-i} m^j = \\
|V(F_{m,d})|-2 \frac{m^{d+1-i}-1}{m-1}. 
\end{align*}

Since there are $m^i$ edges in $F_{m,d}$ between a vertex at distance $i$ from the root and a vertex at distance $i-1$ from the root, we obtain 

\begin{align*}
\mo(F_{m, d}) =  \\
\sum_{i = 1}^{d} m^i(|V(F_{m,d})|-2\frac{m^{d+1-i}-1}{m-1}) = \\
|V(F_{m,d})|^2 - |V(F_{m,d})| - 2 \sum_{i = 1}^{d} \frac{m^{d+1}-m^i}{m-1} = \\
|V(F_{m,d})|^2 - |V(F_{m,d})| - 2d(|V(F_{m,d})|+\frac{1}{m-1})+\frac{2}{m-1}(|V(F_{m,d})|-1).
\end{align*}
\end{proof}

\begin{cor}
Among all full $m$-ary trees $F$ of order $n$, the maximum possible value of $\mo(F)-\irr(F)$ is $n^2-\Theta(\frac{\log{n}}{\log{\log{n}}}n)$.
\end{cor}

\begin{proof}
Let $F = F_{m, d}$, so that $n = \frac{m^{d+1}-1}{m-1} = \Theta(m^d)$. Then $\mo(F_{m,d}) = n^2 - n - 2d(n + \frac{1}{m-1})+\frac{2}{m-1}(n-1) = n^2 -\Theta(d n)$ and $\irr(F) = m+m^{d+1} = \Theta(m n)$. Thus $\mo(F_{m, d})-\irr(F_{m,d}) = n^2-\Theta(mn+dn)$. If $m = \Theta(\frac{\log{n}}{\log{\log{n}}})$, then $d = \Theta(\frac{\log{n}}{\log{m}}) = \Theta(\frac{\log{n}}{\log{\log{n}}}) = \Theta(m)$, so $\mo(F_{m, d})-\irr(F_{m,d}) = n^2-\Theta(\frac{\log{n}}{\log{\log{n}}}n)$. Thus the maximum possible value of $\mo(F)-\irr(F)$ among all full $m$-ary trees $F$ of order $n$ is $n^2-O(\frac{\log{n}}{\log{\log{n}}}n)$.

To see that the maximum possible value of $\mo(F)-\irr(F)$ among all full $m$-ary trees $F$ of order $n$ is $n^2-\Omega(\frac{\log{n}}{\log{\log{n}}}n)$, it suffices to show that $m+d = \Omega(\frac{\log{n}}{\log{\log{n}}})$ since $\mo(F_{m, d})-\irr(F_{m,d}) = n^2-\Theta(mn+dn)$.

Since $d = \Theta(\frac{\log{n}}{\log{m}})$, we have $m+d = \Omega(m+\frac{\log{n}}{\log{m}})$. We split into two cases. If $m \ge \frac{\log{n}}{\log{\log{n}}}$, then we have $m+d = \Omega(\frac{\log{n}}{\log{\log{n}}})$. Otherwise if $m < \frac{\log{n}}{\log{\log{n}}}$, then $\log{m} < \log{\log{n}}$, so $\frac{\log{n}}{\log{m}} = \Omega(\frac{\log{n}}{\log{\log{n}}})$ and $m+d = \Omega(\frac{\log{n}}{\log{\log{n}}})$.
\end{proof}

\subsection{Balanced spiders}

Let $S_{a, b}$ denote the balanced spider graph of order $ab+1$ with $a$ legs of length $b$. We determine the exact value of $\irr(S_{a, b})$ and $\mo(S_{a, b})$ for all $a, b \ge 2$.

For $b = 1$, it is already known that $\irr(S_{a, 1}) = \mo(S_{a,1}) = a(a-1)$  \cite{albertson,mostar1,mo-irr}.

\begin{thm}
For all $a \ge 2$ and $b \ge 1$, we have
\begin{enumerate}
\item $\irr(S_{a, b}) = a(a-1)$
\item $\mo(S_{a, b}) = a^2b^2 - a b^2$
\end{enumerate}
\end{thm}

\begin{proof}
To see that $\irr(S_{a, b}) = a(a-1)$, note that $d_v = 2$ for all vertices in $S_{a, b}$ except the center and the leaves, so $|d_u-d_v| = 0$ for all edges $\left\{u, v\right\} \in E(S_{a, b})$ except for edges that contain the center and edges that contain leaves. If $\left\{u, v\right\}$ contains the center, then $|d_u-d_v| = a-2$. There are $a$ edges that contain the center, so these edges contribute $(a-2)a$ to the sum. Moreover, any edge $\left\{u, v\right\} \in E(S_{a, b})$ that contains a leaf satisfies $|d_u-d_v| = 1$, and there are $a$ edges that contain leaves, so these edges contribute $a$ to the sum. Thus $\irr(S_{a, b}) = a(a-1)$.

To see that $\mo(S_{a, b}) = a^2b^2 - a b^2$, we determine $|n_{S_{a,b}}(u, v)-n_{S_{a,b}}(v, u)|$ for each edge in $E(S_{a,b})$. If $\left\{u,v\right\}$ is between vertices $u$ and $v$ that are distances $i$ and $i-1$ from the center respectively, then $n_{S_{a,b}}(u, v) = b+1-i$ and $n_{S_{a,b}}(v, u) = ab-b+i$. Thus $|n_{S_{a,b}}(u, v)-n_{S_{a,b}}(v, u)| = ab-2b+2i-1$.

Therefore we have $\mo(S_{a, b}) = a \sum_{i = 1}^b (ab-2b+2i-1) = ab (ab-2b-1) + a b (b+1) = a^2b^2 - a b^2$.
\end{proof}

\begin{cor}
Among all balanced spiders $S$ of order $n$, the maximum possible value of $\mo(S)-\irr(S)$ is $n^2-\Theta(n^{4/3})$.
\end{cor}

\begin{proof}
Let $S = S_{a,b}$ and define $m = n-1 = ab$. Then $\mo(S_{a, b})-\irr(S_{a,b}) = a^2b^2 - a b^2 - a(a-1) = n^2(1-\frac{1}{a})-a^2+a$. Let $f(x) = n^2(1-\frac{1}{x})-x^2+x$.

If we let $x = \Theta(n^{2/3})$, then $f(x) = n^2 - \Theta(n^{4/3})$. Thus the maximum possible value of $f(x)$ is $n^2 - O(n^{4/3})$. 

For the upper bound, note that $\frac{n^2}{2x}+\frac{n^2}{2x}+x^2 \ge 3(\frac{n^4}{4})^{1/3}$ for $x \in [0,n]$ by the arithmetic mean geometric mean inequality, so 

$$\frac{n^2}{2x}+\frac{n^2}{2x}+x^2-x \ge 3(\frac{n^4}{4})^{1/3}-x \ge 3(\frac{n^4}{4})^{1/3}-n.$$

Thus $f(x) = n^2-\Omega(n^{4/3})$ for all $x \in [0,n]$.
\end{proof}

\section{Total Mostar index}\label{totalmostar}

In the next theorem, we generalize the result of Kramer and Rautenbach \cite{KrRa} that the maximum possible value of $\mo^{*}(T)$ over all trees $T$ of order $n$ is $\frac{1}{2}n^3(1-o(1))$. Another way to state their result is that the maximum possible value of $\mo^{*}(T)$ over all graphs $T$ of order $n$ and degeneracy $1$ is $\frac{1}{2}n^3(1-o(1))$. We show that degeneracy $1$ can be replaced with degeneracy $k$, and we still obtain the same bound up to an $o(n^3)$ additive factor.

In order to prove this result, we define a family of graphs which generalize the balanced spiders. Define the $k$-thick balanced spider $S_{a, b, k}$ to be the graph obtained from $S_{a, b}$ by adding an edge between any two vertices on the same leg that are at most a distance of $k$ apart in $S_{a, b}$. Note that in this definition, we do not consider the center vertex to be on any of the legs.

Observe that for $a \ge k$ and $b \ge k+1$, every vertex in $S_{a, b, k}$ has degree at least $k$, so $S_{a,b,k}$ has degeneracy at least $k$. On the other hand, any subgraph $H$ of $S_{a, b, k}$ must have a vertex of degree at most $k$ in $H$. We can take this vertex to be a farthest leg vertex from the center in $H$ if there is any leg vertex in $H$, and otherwise the center has degree $0$ in $H$ if there are no leg vertices in $H$. Thus $S_{a,b,k}$ has degeneracy exactly $k$ for $a \ge k$ and $b \ge k+1$.

\begin{thm}
Among all graphs $G$ of order $n$ and degeneracy $k$, the maximum possible value of $\mo^{*}(G)$ is $\frac{1}{2}n^3-O(n^{5/2}+k^2 n^2)$.
\end{thm}

\begin{proof}
The upper bound of $\binom{n}{2}n$ is immediate from the definition of $\mo^{*}(G)$. For the lower bound, consider the graph $G = S_{m,m,k}$ with $m \ge k+1$, and let $n = m^2+1$.

Suppose that $\left\{u,v\right\} \subset V(G)$. If one of $u$ or $v$ is the center, then $|n_G(u)-n_G(v)| \ge m(m-2)$. For the remaining cases, suppose that both $u$ and $v$ are leg vertices.

If $u$ and $v$ are on the same leg, then $|n_G(u)-n_G(v)| \ge m(m-2)$ unless both $u$ and $v$ are within distance $k+1$ of the center in $S_{m,m}$. The number of subsets $\left\{u,v\right\}$ for which $u$ and $v$ are on the same leg and both within distance $k+1$ of the center is $m \binom{k+1}{2}$.

If $u$ and $v$ are on different legs, then $|n_G(u)-n_G(v)| \ge m(m-2)$ unless $u$ and $v$ have the same distance to the center in $G$. The only way that $u$ and $v$ have the same distance to the center in $G$ is if they have the same distance to the center in $S_{m, m}$ or if they are both within distance $k+1$ of the center in $S_{m, m}$.

The number of subsets $\left\{u,v\right\}$ for which $u$ and $v$ are on different legs and have the same distance to the center in $S_{m, m}$ is equal to $m \binom{m}{2}$. The number of subsets $\left\{u,v\right\}$ for which $u$ and $v$ are on different legs and are both within distance $k+1$ of the center in $S_{m, m}$ is equal to $(k+1)^2\binom{m}{2}$.

Thus we have shown that for all but at most $m \binom{k+1}{2}+m \binom{m}{2}+(k+1)^2\binom{m}{2} = O(k^2 n +n^{3/2})$ subsets $\left\{u,v\right\} \subset V(G)$, we have $|n_G(u)-n_G(v)| \ge m(m-2) = n - O(\sqrt{n})$. Therefore 
\begin{align*}
\mo^{*}(G) \ge  \\
(\binom{n}{2}-O(k^2 n+n^{3/2}))m(m-2) = \\
(\binom{n}{2}-O(k^2 n+n^{3/2}))(n-O(\sqrt{n})) = \\
\frac{1}{2}n^3 - O(n^{5/2}+k^2 n^2).
\end{align*}
\end{proof}

In the last proof, note that the construction was for $n$ of the form $n = m^2+1$. In order to obtain the bound of $\frac{1}{2}n^3-O(n^{5/2}+k^2 n^2)$ for any $n$, observe that we can modify the construction in the last proof slightly by allowing the lengths of the legs of the $k$-thick spider to be at most $1$ apart instead of requiring them to be all equal.

\begin{cor}
Among all graphs $G$ of order $n$ and degeneracy $k$ with $k = o(\sqrt{n})$, the maximum possible value of $\mo^{*}(G)$ is $\frac{1}{2}n^3(1-o(1))$.
\end{cor}

We complement the asymptotically sharp bounds on general spiders in \cite{KrRa} by finding the exact value of the total Mostar index for all balanced spiders. Note that the result below gives the value of $\mo^{*}(P_n)$ for odd $n$ when $a = 2$ and $b = \frac{n-2}{2}$, and the value of $\mo^{*}(K_{1,n-1})$ for any $n$ when $a = n-1$ and $b = 1$

\begin{thm}
For $a \ge 2$, $b \ge 1$ and $n = 1+ a b$, $\mo^{*}(S_{a, b}) = \frac{(n-1)(3ab n -5 a b^2 - 3 a n + 3 a b+2a+2b^2-9b+6n-5)}{6}$.
\end{thm}

\begin{proof}
In order to determine $\mo^{*}(S_{a, b})$, we first calculate $|n_{S_{a,b}}(u, v)-n_{S_{a,b}}(v, u)|$ for each $\left\{u, v\right\} \subset V(S_{a, b})$. Note that if $u$ and $v$ have the same distance to the center of $S_{a, b}$, then $|n_{S_{a,b}}(u, v)-n_{S_{a,b}}(v, u)| = 0$. 

Without loss of generality, suppose that $u$ has distance $i$ to the center, $v$ has distance $j$ to the center, and $i < j$. We split into three cases. 

For the first case, we assume that $u$ is the center. If $i$ is odd, then $n_{S_{a, b}}(v, u) = b- \frac{i-1}{2}$ and $n_{S_{a, b}}(u, v) = n-b+ \frac{i-1}{2}$, so  $|n_{S_{a,b}}(u, v)-n_{S_{a,b}}(v, u)| = n-2b+i-1$. If $i$ is even, then $n_{S_{a, b}}(v, u) = b- \frac{i}{2}$ and $n_{S_{a, b}}(u, v) = n-b+ \frac{i}{2}-1$, so  $|n_{S_{a,b}}(u, v)-n_{S_{a,b}}(v, u)| = n-2b+i-1$.

For the second case, we assume that $u$ and $v$ are on the same leg with $i > 0$.  If $i+j$ is odd, then $n_{S_{a, b}}(v, u) = b- \frac{i+j-1}{2}$ and $n_{S_{a, b}}(u, v) = n-b+ \frac{i+j-1}{2}$, so  $|n_{S_{a,b}}(u, v)-n_{S_{a,b}}(v, u)| = n-2b+i+j-1$. If $i+j$ is even, then $n_{S_{a, b}}(v, u) = b- \frac{i+j}{2}$ and $n_{S_{a, b}}(u, v) = n-b+ \frac{i+j}{2}-1$, so  $|n_{S_{a,b}}(u, v)-n_{S_{a,b}}(v, u)| = n-2b+i+j-1$.

For the third case, we assume that $u$ and $v$ are on different legs with $i > 0$.  If $j-i$ is odd, then $n_{S_{a, b}}(v, u) = b- \frac{j-i-1}{2}$ and $n_{S_{a, b}}(u, v) = n-b+ \frac{j-i-1}{2}$, so  $|n_{S_{a,b}}(u, v)-n_{S_{a,b}}(v, u)| = n-2b+j-i-1$. If $j-i$ is even, then $n_{S_{a, b}}(v, u) = b- \frac{j-i}{2}$ and $n_{S_{a, b}}(u, v) = n-b+ \frac{j-i}{2}-1$, so  $|n_{S_{a,b}}(u, v)-n_{S_{a,b}}(v, u)| = n-2b+j-i-1$. Thus 

\footnotesize
\begin{align*}
\mo^{*}(S_{a, b}) =  \\
\left( \sum_{1 \le i < j \le b} a(n-2b+i+j-1)+a(a-1)(n-2b+j-i-1) \right)+ a \sum_{1 \le i \le b} (n-2b+i-1) = \\
\left( \sum_{1 \le j \le b} a(n-2b+j-1)(j-1) + a\frac{j(j-1)}{2}+a(a-1)(n-2b+j-1)(j-1)-a(a-1)\frac{j(j-1)}{2} \right) + \\
a b (n - 2b-1) + a b (\frac{b+1}{2}) = \\
\frac{1}{2}a^2 b^2 n - \frac{5}{6} a^2 b^3-\frac{a^2 b n}{2} + \frac{a^2 b^2}{2} + \frac{a^2 b+a b^3}{3}-\frac{3 a b^2}{2} + a b n -\frac{5ab}{6} = \\
\frac{(n-1)(3ab n -5 a b^2 - 3 a n + 3 a b+2a+2b^2-9b+6n-5)}{6}.
\end{align*}
\normalsize

\end{proof}

In the last proof, note that we used the fact that $a \ge 2$ and $b \le \frac{n-1}{2}$ when we concluded that $|n_{S_{a,b}}(u, v)-n_{S_{a,b}}(v, u)| = n-2b+i-1$ in the first case, $|n_{S_{a,b}}(u, v)-n_{S_{a,b}}(v, u)| = n-2b+i+j-1$ in the second case, and $|n_{S_{a,b}}(u, v)-n_{S_{a,b}}(v, u)| = n-2b+j-i-1$ in the third case.

\begin{cor}
Among all balanced spiders $S$ of order $n$, the maximum possible value of $\mo^{*}(S)$ is $\frac{1}{2}n^3-\Theta(n^{5/2})$.
\end{cor}

\begin{proof}
Note that the maximum does not occur when $S$ is a star, so we may assume that $b \ge 2$.

Let $S = S_{a,b}$ with $a \ge 2$ and $b \ge 2$, so $n = 1+a b$ and 
$$\mo^{*}(S) = \frac{(n-1)(3ab n -5 a b^2 - 3 a n + 3 a b+2a+2b^2-9b+6n-5)}{6} = \frac{(n-1)(3n^2-\Theta(nb+na))}{6}.$$ 

To see that the maximum possible value of $\mo^{*}(S)$ among all balanced spiders $S$ of order $n$ is $\frac{1}{2}n^3-O(n^{5/2})$, note that we can take $b = \Theta(\sqrt{n})$, in which case $\mo^{*}(S) = \frac{1}{2}n^3-O(n^{5/2})$.

To see that the maximum possible value of $\mo^{*}(S)$ among all balanced spiders $S$ of order $n$ is $\frac{1}{2}n^3-\Omega(n^{5/2})$, note that $nb + na \ge 2 n \sqrt{a b} = \Theta(n^{3/2})$, so $\frac{(n-1)(3n^2-\Theta(nb+na))}{6} = \frac{1}{2}n^3-\Omega(n^{5/2})$.
\end{proof}

\section{Extremal bounds and exact values for $\spr(G)$}\label{sprres}

In this section, we determine the exact value of $\spr(G)$ when $G$ is a path, cycle, complete graph, complete bipartite graph, and balanced spider. We prove an upper bound on $\spr(G)$ in terms of the order and diameter of $G$. We show that the maximum possible value of $\spr(G)$ among all connected graphs $G$ of order $n$ is $\frac{1}{2}n^3(1-o(1))$. We show that the minimum possible value of $\spr(G)$ among all graphs $G$ of order $n$ is $n^2-n$, and the same is true for connected graphs. We also prove a number of results about $\spr(v)$ for vertices $v$ in various families of graphs.

We start by determining the sum peripherality of paths.

\begin{thm}\label{sprpn}
$\spr(P_n)$ is $\frac{1}{2}n^2(n-1)-\frac{n}{2}(\frac{n}{2}-1)$ when $n$ is even and $\frac{1}{2}n^2(n-1)-\frac{(n-1)^2}{4}$ otherwise.
\end{thm}

\begin{proof}
To compute $\spr(P_n)$, we subtract $\sum_{\{u,v\}\subset V(P_n), w\in V(P_n)}\mathbbm{1}{\left[d(w,u)=d(w,v)\right]}$ from $\frac{1}{2}n^2(n-1)$. Let the vertices of $P_n$ be $v_1, \dots, v_n$ in path order.

Suppose that $n$ is even. If $w = v_i$ for $1 \le i \le \frac{n}{2}$, then there are $i-1$ choices for $\left\{u, v\right\}$. Similarly if $w = v_{n+1-i}$ for $1 \le i \le \frac{n}{2}$, then there are $i-1$ choices for $\left\{u, v\right\}$. Thus in this case $\sum_{\{u,v\}\subset V(P_n), w\in V(G)}\mathbbm{1}{\left[d(w,u)=d(w,v)\right]} = 2 (\frac{1}{2})(\frac{n}{2})(\frac{n}{2}-1) = \frac{n}{2}(\frac{n}{2}-1)$.

Suppose that $n$ is odd. If $w = v_i$ for $1 \le i \le \frac{n-1}{2}$, then there are $i-1$ choices for $\left\{u, v\right\}$. Similarly if $w = v_{n+1-i}$ for $1 \le i \le \frac{n-1}{2}$, then there are $i-1$ choices for $\left\{u, v\right\}$. If $w = v_{\frac{n+1}{2}}$, then there are $\frac{n-1}{2}$ choices for $\left\{u, v\right\}$. Thus in this case $\sum_{\{u,v\}\subset V(P_n), w\in V(G)}\mathbbm{1}{\left[d(w,u)=d(w,v)\right]} = 2 (\frac{1}{2})(\frac{n-1}{2})(\frac{n-1}{2}-1)+\frac{n-1}{2} = \frac{(n-1)^2}{4}$.
\end{proof}

In the last result, we computed $\spr(P_n)$ without determining $\spr(v)$ for any vertices $v$ in $P_n$. In the next result, we determine $\spr(v)$ for all vertices $v$ in $P_n$, and we use this to show that the value of $\spr(v)$ for $v \in V(P_n)$ increases with the distance of $v$ to the center.

\begin{thm}\label{sprvpn}
Let the vertices of $P_n$ be $v_1, \dots, v_n$ in path order. Then we have
\begin{align*}
\spr(v_i) =  \frac{3}{2}i^2-\frac{3}{2}(1+n)i+1+ \frac{3n^2}{4} \text{ if n is even} , \\
\spr(v_i) =  \frac{3}{2}i^2-\frac{3}{2}(1+n)i+1+ \frac{3n^2+1}{4} \text{ if i is even and n is odd},\\
\spr(v_i) =  \frac{3}{2}i^2-\frac{3}{2}(1+n)i+\frac{3}{4}+ \frac{3n^2}{4}  \text{ if i is odd and n is odd}.
\end{align*}
\end{thm}

\begin{proof}
If $j <i$, then $n_{P_n}(v_j,v_i) = \lceil \frac{i+j}{2} \rceil -1$. If $j > i$, then $n_{P_n}(v_j,v_i) = n-\lfloor \frac{i+j}{2} \rfloor$.

Suppose that $i$ is even. Then $\sum_{j = 1}^{i-1} n_{P_n}(v_j,v_i) = \sum_{j = 1}^{i-1} ( \lceil \frac{i+j}{2} \rceil -1 ) = (i-2)\frac{\frac{i}{2}+1+i}{2}+\frac{i}{2}+1 -(i-1) = \frac{3}{4}i^2 - \frac{3}{2} i + 1$. If $n$ is even, then $\sum_{j=i+1}^{n} (n-\lfloor \frac{i+j}{2} \rfloor) = (n-i)\frac{n-i+n-\frac{n+i}{2}}{2} =  \frac{3}{4}i^2 - \frac{3n}{2} i + \frac{3n^2}{4}$, so $\spr(v_i) = \frac{3}{2}i^2-\frac{3}{2}(1+n)i+1+ \frac{3n^2}{4}$. If $n$ is odd, then $\sum_{j=i+1}^{n} (n-\lfloor \frac{i+j}{2} \rfloor) = (n-i-1)\frac{n-i+n-\frac{n+i-1}{2}}{2} + n-\frac{n+i-1}{2} =  \frac{3}{4}i^2 - \frac{3n}{2} i + \frac{3n^2+1}{4}$, so $\spr(v_i) = \frac{3}{2}i^2-\frac{3}{2}(1+n)i+1+ \frac{3n^2+1}{4}$.

Suppose that $i$ is odd. Then $\sum_{j = 1}^{i-1} n_{P_n}(v_j,v_i) = \sum_{j = 1}^{i-1} ( \lceil \frac{i+j}{2} \rceil -1 ) = (i-1)\frac{\frac{i+1}{2}+i}{2} -(i-1) = \frac{3}{4}i^2 - \frac{3}{2}i + \frac{3}{4}$. If $n$ is even, then $\sum_{j=i+1}^{n} (n-\lfloor \frac{i+j}{2} \rfloor) = (n-i-1)\frac{n-i+n-\frac{n+i-1}{2}}{2}+ n-\frac{n+i-1}{2} = \frac{3}{4}i^2 - \frac{3n}{2} i + \frac{3n^2+1}{4}$, so 
$$\spr(v_i) = \frac{3}{2}i^2-\frac{3}{2}(1+n)i+\frac{3}{4}+ \frac{3n^2+1}{4} = \frac{3}{2}i^2-\frac{3}{2}(1+n)i+1+ \frac{3n^2}{4}.$$
If $n$ is odd, then $\sum_{j=i+1}^{n} (n-\lfloor \frac{i+j}{2} \rfloor) = (n-i)\frac{n-i+n-\frac{n+i}{2}}{2}  =  \frac{3}{4}i^2 - \frac{3n}{2} i + \frac{3n^2}{4}$, so $\spr(v_i) = \frac{3}{2}i^2-\frac{3}{2}(1+n)i+\frac{3}{4}+ \frac{3n^2}{4}$.
\end{proof}

\begin{cor}
The value of $\spr(v)$ for $v \in V(P_n)$ increases with the distance of $v$ to the center of $P_n$.
\end{cor}

\begin{proof}
Note that for each case in Theorem~\ref{sprvpn}, $\frac{d \spr(v_i)}{d i} = 3i - \frac{3}{2}(1+n)$, which is negative for $i < \frac{1+n}{2}$ and positive for $i > \frac{1+n}{2}$. Note that for both even and odd values of $n$, the formulas for $\spr(v_i)$ for odd $i$ and even $i$ differ by at most $\frac{1}{2}$ for any fixed value of $i$. 

Moreover for both even and odd values of $n$, the formula for $\spr(v_i)$ for even $i$ has integer values for even values of $i$, and the formula for $\spr(v_i)$ for odd $i$ has integer values for odd values of $i$. Thus the value of $\spr(v_i)$ increases with $|i-\frac{1+n}{2}|$.
\end{proof}

In the next result, we determine the sum peripherality of cycles.

\begin{thm}
$\spr(C_n)$ is $\frac{1}{2}n^2(n-1)-\frac{n(n-2)}{2}$ when $n$ is even and $\frac{1}{2}n^2(n-1)-\frac{n(n-1)}{2}$ otherwise.
\end{thm}

\begin{proof}
To compute $\spr(C_n)$, we subtract $\sum_{\{u,v\}\subset V(C_n), w\in V(C_n)}\mathbbm{1}{\left[d(w,u)=d(w,v)\right]}$ from $\frac{1}{2}n^2(n-1)$. Let the vertices of $C_n$ be $v_1, \dots, v_n$ in cyclic order.

Suppose that $n$ is even. For every choice of $w$, there are $\frac{n-2}{2}$ choices for $\{u,v\}$. Thus in this case $\sum_{\{u,v\}\subset V(P_n), w\in V(G)}\mathbbm{1}{\left[d(w,u)=d(w,v)\right]} = \frac{n(n-2)}{2}$.

Suppose that $n$ is odd. For every choice of $w$, there are $\frac{n-1}{2}$ choices for $\{u,v\}$. Thus in this case $\sum_{\{u,v\}\subset V(P_n), w\in V(G)}\mathbbm{1}{\left[d(w,u)=d(w,v)\right]} = \frac{n(n-1)}{2}$.\end{proof}

We use the next theorem to obtain an upper bound on the sum peripherality in terms of the order and the diameter.

\begin{thm}
\label{ecc}
For any vertex $w$ in $G$ with eccentricity $\epsilon(w)$,
$$
\sum_{\{u,v\}\subset V(G)}\mathbbm{1}{\left[d(w,u)=d(w,v)\right]}\geq
\frac{1}{2}\left(\frac{(n-1)^2}{\epsilon(w)}-(n-1)\right).
$$
\end{thm}
\begin{proof}
Let $k_j$ be the number of vertices at distance $j$ from $w$.
\begin{align*}
\sum_{\{u,v\}\subset V(G)}\mathbbm{1}{\left[d(w,u)=d(w,v)\right]}= \\
\binom{k_1}{2}+\ldots+\binom{k_{\epsilon(w)}}{2}=\\
\frac{k_1(k_1-1)}{2}+\ldots+\frac{k_{\epsilon(w)}(k_{\epsilon(w)}-1)}{2}= \\
\frac{1}{2}\sum_{j=1}^{\epsilon(w)}k_j^2-\frac{n-1}{2}\ge \\
\frac{1}{2}\left(\frac{(n-1)^2}{\epsilon(w)}-(n-1)\right),
\end{align*}
\noindent where the last inequality follows by the Cauchy-Schwartz inequality.
\end{proof}

The next theorem gives an upper bound on the sum peripherality of any graph $G$ in terms of its order and diameter, and we will use this result to asymptotically determine the maximum possible value of $\spr(G)$ over all connected graphs $G$ of order $n$.

\begin{thm}
$\spr(G)\le \frac{1}{2}n^2(n-1)-\frac{n(n-1)^2}{2D}+\frac{n(n-1)}{2}$ for any graph $G$ with diameter $\diam(G)\le D$.
\end{thm}
\begin{proof}
To compute $\spr(G)$, we subtract $\sum_{\{u,v\}\subset V(G), w\in V(G)}\mathbbm{1}{\left[d(w,u)=d(w,v)\right]}$ from $\frac{1}{2}n^2(n-1)$. Since every vertex $w\in V(G)$ has eccentricity bounded by $D$, the result follows from Theorem~\ref{ecc}.
\end{proof}


\begin{cor}
\label{cor:2n/3}
A graph of order $n$ with $\diam(G)<\frac{2(n-1)}{3}$ does not maximize $\spr(G)$ among graphs $G$ of order $n$.
\end{cor}
\begin{proof}
This is because for $D<\frac{2(n-1)}{3}$, 
$$
\frac{n}{2}\left(\frac{(n-1)^2}{D}-(n-1)\right) > \frac{1}{4}n(n-1) > \max\left(\frac{n}{2}(\frac{n}{2}-1), \frac{(n-1)^2}{4}\right).
$$
\end{proof}

Next, we show that the maximum possible value of $\spr(G)$ over all connected graphs $G$ of order $n$ is $\frac{1}{2}n^3(1-o(1))$.

\begin{thm}
The maximum possible value of $\spr(G)$ for all connected graphs $G$ of order $n$ is bounded by $\frac{1}{2}n^3-\left(\ln{\sqrt{3}}\pm o(1)\right)n^2\approx \frac{1}{2}n^3-0.55n^2$ from above.
\end{thm}
\begin{proof}
Let $T$ be a spanning tree of $G$. For any vertex $w\in V(G)$, $\epsilon_T(w)\ge\epsilon(w)$.
By Theorem~\ref{ecc}

\begin{align*}
\spr(G)&\leq\frac{1}{2}n^2(n-1)+\frac{n(n-1)}{2}-\frac{(n-1)^2}{2}\sum_{w\in V(G)}{\frac{1}{\epsilon(w)}}\\
       &\leq\frac{1}{2}n^2(n-1)+\frac{n(n-1)}{2}-\frac{(n-1)^2}{2}\sum_{w\in V(G)}{\frac{1}{\epsilon_T(w)}}.
\end{align*}

Let $u$ be a center of $T$, i.e. $\epsilon_T(u)\le \frac{n}{2}$. Suppose $u$ has $k$ branches with $n_1$, $n_2$, \ldots, $n_k$ vertices, respectively. As we traverse within each branch in Breadth-First order starting from $u$, the eccentricity of the traversed vertex in $T$ changes by at most $1$ whenever we go one level down the tree. Therefore
\begin{align*}
\sum_{w\in V(G)}{\frac{1}{\epsilon_T(w)}}&\ge \frac{1}{\frac{n}{2}}+\sum_{j=1}^k \sum_{i = 1}^{n_j} \frac{1}{\frac{n}{2}+i},
\end{align*}
\noindent where $n_1+n_2+\ldots+n_k=n-1$.
The right hand side of the above is minimized when $k=1$ and  $n_1=n-1$. Therefore
\begin{align*}
\spr(G)&\leq\frac{1}{2}n^2(n-1)+\frac{n(n-1)}{2}-\frac{(n-1)^2}{2}\sum_{w\in V(G)}{\frac{1}{\epsilon_T(w)}}\\
&\leq \frac{1}{2}n^3-\frac{1}{2}\left(\left(\sum_{i = 0}^{n-1} \frac{1}{\frac{n}{2}+i}\right)\pm o(1)\right)n^2\\
&=\frac{1}{2}n^3-\frac{1}{2}\left(\ln{3}\pm o(1)\right)n^2.
\end{align*}
\end{proof}

In the next result, we determine $\spr(v)$ for each vertex $v$ in a complete graph.

\begin{thm}
For all vertices $v$ in $K_n$, we have $\spr(v) = n-1$.
\end{thm}

\begin{proof}
Given $v$, we have $n-1$ choices for $u \in V(K_n)-v$, and $n_{K_n}(u, v) = 1$ for all $u \in V(K_n)-v$.
\end{proof}

We obtain the next result by multiplying the last result by $n$.

\begin{thm}\label{spr:kn}
$\spr(K_n) = n^2-n$ for all $n \ge 2$.
\end{thm}

\begin{proof}
We have $n$ choices for $v$, and $\spr(v) = n-1$ for all vertices $v$ in $K_n$.
\end{proof}

Using $K_n$, we determine the minimum possible value of $\spr(G)$ among all graphs $G$ of order $n$.

\begin{thm}
The minimum possible value of $\spr(G)$ among all graphs $G$ of order $n$ is $n^2-n$.
\end{thm}

\begin{proof}
We obtain the lower bound $\spr(G) \ge n^2-n$ for all graphs $G$ of order $n$ since $n_G(u, v) \ge 1$ for all $u \in V(G)-v$, and there are $n$ choices for $v$ and $n-1$ choices for $u$.

The upper bound follows from Theorem~\ref{spr:kn}.
\end{proof}

Next we determine the exact value of the sum peripherality of every balanced spider.

\begin{thm}
If $n = 1+ a b$ with $a \ge 2$ and $b \ge 1$, then $\spr(S_{a,b}) = \frac{1}{2}n^2(n-1) - \frac{ab^2(2a^2-5a+4)}{4}$ when $b$ is even and $\spr(S_{a,b}) = \frac{1}{2}n^2(n-1)-(\frac{a^3 b^2}{2}-\frac{5a^2 b^2}{4}+\frac{a^2}{4}+ab^2-\frac{a}{2})$ when $b$ is odd.
\end{thm}

\begin{proof}
To compute $\spr(S_{a,b})$, we subtract $\sum_{\{u,v\}\subset V(S_{a,b}), w\in V(S_{a,b})}\mathbbm{1}{\left[d(w,u)=d(w,v)\right]}$ from $\frac{1}{2}n^2(n-1)$. Let the vertices of $S_{a,b}$ consist of the center vertex $c$ and the leg vertices $v_{i,j}$ with $1 \le i \le a$ and $1 \le j \le b$ where the vertices $v_{i, 1}, v_{i, 2}, \dots, v_{i,b}$ are on the same leg for each $i$ and $d(c, v_{i, j}) = j$ for all $i, j$.

If $u = v_{i,j}$ and $v = v_{i',j'}$, note that $w$ exists and is uniquely determined if $j$ and $j'$ have the same parity and $j \neq j'$. If $j = j'$, then there are $1+(a-2)b$ choices for $w$. 

If $b$ is even, then we have 

\begin{align*}
\sum_{\{u,v\}\subset V(S_{a,b}), w\in V(S_{a,b})}\mathbbm{1}{\left[d(w,u)=d(w,v)\right]} = \\
(a-2)b\binom{a}{2}b+\binom{1+\frac{ab}{2}}{2}+\binom{\frac{ab}{2}}{2} = \\
\frac{ab^2(2a^2-5a+4)}{4}.
\end{align*}

If $b$ is odd, then we have 

\begin{align*}
\sum_{\{u,v\}\subset V(S_{a,b}), w\in V(S_{a,b})}\mathbbm{1}{\left[d(w,u)=d(w,v)\right]} = \\
(a-2)b\binom{a}{2}b+\binom{1+\frac{a(b-1)}{2}}{2}+\binom{\frac{a(b+1)}{2}}{2} = \\
\frac{a^3 b^2}{2}-\frac{5a^2 b^2}{4}+\frac{a^2}{4}+ab^2-\frac{a}{2}.
\end{align*}
\end{proof}

For balanced spiders, we show that $\spr(v)$ acts as a measure of peripherality of the vertices. In other words, it is higher for vertices that are farther from the center.

\begin{thm}
For $a \ge 3$ and $b \ge 1$, the value of $\spr(v)$ increases with the distance of $v$ from the center in $S_{a,b}$.
\end{thm}

\begin{proof}
Let $v$ have distance $j > 0$ to the center.

We consider each possibility for $u \in V(S_{a,b})-v$. If $u$ is on the same leg as $v$ with distance $i > j$ from the center, then $n_G(u,v) = b-\lfloor \frac{i+j}{2} \rfloor$.

If $u$ is on the same leg as $v$ with distance $i < j$ from the center, then $n_G(u,v) = (a-1)b+\lceil \frac{i+j}{2} \rceil$.

If $u$ is the center vertex, then $n_G(u,v) = (a-1)b+ \lceil \frac{j}{2} \rceil$.

If $u$ is on a different leg from $v$ with distance $i > j$ from the center, then $n_G(u,v) = b - \lfloor \frac{i-1-j}{2} \rfloor$.

If $u$ is on a different leg from $v$ with distance $i < j$ from the center, then $n_G(u,v) = (a-1)b+ \lceil \frac{j-i-1}{2} \rceil$.

If $u$ is on a different leg from $v$ with distance $j$ from the center, then $n_G(u,v) = b$.

For any $v$ with distance $j > 0$ to the center, let $f(j)$ denote the sum of $n_G(u,v)$ over all $u \in V(S_{a,b})-v$. Observe that if we increase $j$ to $j+1$, the only terms in the sum that decrease are for $u$ on the same leg with distance $i > j$. However there are fewer than $b$ of these terms, they each decrease by at most $1$, and the term for $u$ on the same leg with distance $i = j+1$ in $f(j)$ (which is less than $b$) is replaced with the term for $u$ in the same leg with distance $i = j$ in $f(j+1)$ (which is at least $2b$). 

Thus the terms that decrease only cause the total sum to decrease by less than $b$, but there is an overall increase in the other terms of greater than $b$, so $f(j)$ increases with respect to $j$ for $j > 0$.

If $j = 0$, then $v$ is the center of $S_{a,b}$. Let $v’$ be a vertex adjacent to the center, and consider the sums for $\spr(v)$ and $\spr(v’)$. All of the terms in the sums for each $u \in V(S_{a,b})-\left\{v,v’\right\}$ are non-decreasing when we replace $v$ with $v’$ except for the terms with $u$ on the same leg as $v’$. However each of these terms decreases by at most $1$, and there are fewer than $b$ of these terms. Moreover the term for $v’$ in $\spr(v)$ (which is equal to $b$) is replaced with the term for $v$ in $\spr(v’)$ (which exceeds $2b$). 

Thus the terms that decrease from $\spr(v)$ to $\spr(v’)$ only cause the total sum to decrease by less than $b$, but there is an overall increase in the other terms of greater than $b$, so $\spr(v)$ is minimized when $v$ is the center of $S_{a,b}$.
\end{proof}

In the last two results in this section, we focus on complete bipartite graphs. We start by determining $\spr(v)$ for each vertex $v$ in a complete bipartite graph.

\begin{thm}\label{perikmnv}
In the complete bipartite graph $K_{m,n}$, all vertices $u$ on the left have $\spr(u) = m n+m-1$ and all vertices $v$ on the right have $\spr(v) = m n + n -1$.
\end{thm}

\begin{proof}
If $u$ is any vertex on the left and $v$ is any vertex on the right, then $n_{K_{m,n}}(v,u) = m$ and $n_{K_{m,n}}(u, v) = n$. If $x$ and $y$ are vertices in the same part of $K_{m, n}$, then $n_{K_{m,n}}(x,y) = 1$. Then $\spr(u) = m n+m-1$ and $\spr(v) = m n + n -1$.
\end{proof}

Using the last result, we compute the sum peripherality of every complete bipartite graph.

\begin{cor}
For all $m, n \ge 1$, we have $\spr(K_{m,n}) = m(m n + m-1)+n(m n+ n-1)$. 
\end{cor}

\begin{proof}
There are $m$ vertices $u$ on the left with $\spr(u) = m n+m-1$, and $n$ vertices $v$ on the right with $\spr(v) = m n + n -1$, so we have $\spr(K_{m,n}) = m(m n + m-1)+n(m n+ n-1)$. 
\end{proof}

\section{Extremal bounds and exact values for $\espr(G)$}\label{esprres}

In this section, we determine the edge sum peripherality of complete graphs and complete bipartite graphs. We also prove a general lower bound on $\espr(G)$ in terms of the order of $G$ and the number of edges in $G$. We use this lower bound together with the star $K_{1,n-1}$ to determine the minimum possible value of $\espr(G)$ among all connected graphs $G$ of order $n$.

Unlike $\spr(G)$, there exist graphs $G$ of order $n > 0$ for which $\espr(G) = 0$.

\begin{thm}
The minimum possible value of $\espr(G)$ among all graphs $G$ of order $n$ is $0$.
\end{thm}

\begin{proof}
Any edgeless graph $G$ has $\espr(G) = 0$.
\end{proof}

In the next two results, we determine the edge sum peripherality of each edge in $K_n$, and we use that to determine $\espr(K_n)$.

\begin{thm}
For all edges $e$ in $K_n$, we have $\espr(e) = 2(n-2)$.
\end{thm}

\begin{proof}
Given $e = \left\{u,v\right\}$, we have $n-2$ choices for $x \in V(K_n)-\left\{u,v\right\}$, and $n_{K_n}(x,u)+ n_{K_n}(x,v)= 2$ for all $x \in V(K_n)-\left\{u,v\right\}$.
\end{proof}

We multiply the last result by $\binom{n}{2}$ to obtain $\espr(K_n)$.

\begin{thm}\label{espr:kn}
$\espr(K_n) = n(n-1)(n-2)$ for all $n \ge 2$.
\end{thm}

\begin{proof}
There are $\binom{n}{2}$ edges $e$, and $\espr(e) = 2(n-2)$ for all edges $e$ in $K_n$.
\end{proof}

In the next two results, we determine the edge sum peripherality of each edge in $K_{m,n}$, and we use that to determine $\espr(K_{m,n})$.

\begin{thm}\label{esprkmnv}
In the complete bipartite graph $K_{m,n}$, all edges $e$ have $\espr(e) = 2 m n -2$.
\end{thm}

\begin{proof}
If $u$ is any vertex on the left and $v$ is any vertex on the right, then $n_{K_{m,n}}(v,u) = m$ and $n_{K_{m,n}}(u, v) = n$. If $x$ and $y$ are vertices in the same part of $K_{m, n}$, then $n_{K_{m,n}}(x,y) = 1$. Then for any edge $e$ in $K_{m,n}$, we have $\espr(e) = (m-1)n + m-1 + (n-1)m + n-1 = 2 m n -2$.
\end{proof}

\begin{cor}\label{esprkmn}
For all $m, n \ge 1$, we have $\espr(K_{m,n}) = m n (2 m n - 2)$. 
\end{cor}

\begin{proof}
Every edge $e$ in $K_{m, n}$ has $\espr(e) = 2 m n -2$, and there are $m n$ total edges in $K_{m,n}$, so we have $\espr(K_{m,n}) = m n (2 m n -2)$.  
\end{proof}

Using these results for complete bipartite graphs, we derive some extremal bounds for edge sum peripherality.

\begin{cor}
Among complete bipartite graphs $G$ of order $n$, the maximum possible value of $\espr(G)$ is $\lfloor \frac{n^2}{4} \rfloor (2\lfloor \frac{n^2}{4} \rfloor-2)$.
\end{cor}

\begin{proof}
Suppose that $G$ has one part of size $x$ and the other part of size $n-x$. By Corollary \ref{esprkmn}, we have $\espr(G) = x(n-x)(2x(n-x)-2)$, which is greatest when $x(n-x)$ is maximized, which is when $x = \lfloor \frac{n}{2} \rfloor$ or $x = \lceil \frac{n}{2} \rceil$.
\end{proof}

\begin{cor}
The maximum possible value of $\espr(G)$ among all graphs $G$ of order $n$ is $\Theta(n^4)$.
\end{cor}

\begin{proof}
The lower bound follows from the last corollary. For the upper bound, observe that $\espr(G) < n^4$ since $G$ has at most $\binom{n}{2}$ edges, for each edge $e = \left\{u,v\right\}$ there are $n-2$ vertices $x$ not in $e$, and $n_G(x,u)+n_G(x,v) < 2n$.
\end{proof}

Next we prove a lower bound on $\espr(G)$ with respect to both the order and the number of edges. We note that this is attained by both stars $K_{1, n}$ and complete graphs $K_n$.

\begin{thm}\label{lowerboundorderedge}
For all graphs $G$ of order $n$ with $m$ edges, $\espr(G) \ge 2(n-2)m$.
\end{thm}

\begin{proof}
For every edge $e = \left\{u, v\right\}$, there are $n-2$ choices for a vertex $x \in V(G) - \left\{u,v\right\}$, and $n_G(x,u)+n_G(x,v) \ge 2$. Thus $\espr(G) \ge 2(n-2)m$.
\end{proof}

Using the lower bound we just proved, we determine the minimum possible value of $\espr(G)$ among all connected graphs $G$ of order $n$.

\begin{thm}
The minimum possible value of $\espr(G)$ among all connected graphs $G$ of order $n$ is $2(n-1)(n-2)$.
\end{thm}

\begin{proof}
We obtain the lower bound $\espr(G) \ge 2(n-1)(n-2)$ for all graphs $G$ of order $n$ from Theorem \ref{lowerboundorderedge}, since there are at least $n-1$ edges in $G$. For the matching upper bound, we can use $K_{1, n-1}$.
\end{proof}

\section{Extremal bounds and exact values for $\peri(G)$}\label{perires}

In this section, we determine the minimum and maximum values of $\peri(G)$ among all graphs $G$ and connected graphs $G$ of order $n$. We also determine the exact values of the peripherality of paths, cycles, complete graphs, balanced spiders, and complete bipartite graphs. Moreover, we examine peripherality of vertices in these families of graphs. Furthermore we determine the maximum possible value of $\peri(v)+\deg(v)$ in any graph of order $n$, as well as in any tree of order $n$.

As with $\espr(G)$, it is possible to find graphs $G$ for which $\peri(G) = 0$.

\begin{thm}
The minimum possible value of $\peri(G)$ among all graphs $G$ of order $n$ is $0$.
\end{thm}

\begin{proof}
Any vertex-transitive graph $G$ has $\peri(G) = 0$. More generally, if $G$ is the disjoint union of (possibly different) vertex-transitive graphs $G_1, \dots, G_k$ each of the same order, then $\peri(G) = 0$.
\end{proof}

Since $C_n$ and $K_n$ are vertex-transitive, we obtain the next result.

\begin{thm}
$\peri(C_n) = \peri(K_n) = 0$ for all $n \ge 3$.
\end{thm}

Next we determine the peripherality of paths and of vertices in paths.

\begin{thm}
$\peri(P_n)$ is $\frac{1}{2}n(n-2)$ if $n$ is even and $\frac{1}{2}(n-1)^2$ if $n$ is odd.
\end{thm}

\begin{proof}
To compute $\peri(P_n)$, we subtract $\sum_{\{u,v\}\subset V(P_n)}\mathbbm{1}{\left[n_{P_n}(u,v) = n_{P_n}(v,u)\right]}$ from $\frac{1}{2}n(n-1)$. Let the vertices of $P_n$ be $v_1, \dots, v_n$. The only way that we have $n_{P_n}(u,v) = n_{P_n}(v,u)$ is if $u = v_i$ and $v = v_j$ for some $i, j$ with $i+j = n+1$. The number of subsets $\{v_i,v_j\}\subset V(P_n)$ with $i+j = n+1$ is $\frac{n}{2}$ if $n$ is even and $\frac{n-1}{2}$ if $n$ is odd.
\end{proof}

In the last result, we determined $\peri(P_n)$ without computing the values of $\peri(v)$ for the vertices $v$ in $P_n$. We determine the peripherality of the vertices of $P_n$ in the next result.

\begin{thm}
For $n \ge 3$, let the vertices of $P_n$ be $v_1, \dots, v_n$ in path order. If $n$ is even and $c_1, c_2$ are the center vertices, then $\peri(v) = 2i$ for any vertex $v$ with $\min(d(v,c_1),d(v,c_2)) = i$. If $n$ is odd and $c$ is the center vertex, then $\peri(v) = 2i-1$ for any vertex $v$ with $d(v,c) = i \ge 1$, and $\peri(c) = 0$.
\end{thm}

\begin{proof}
Suppose that $n$ is even, $c_1, c_2$ are the center vertices, and $\min(d(v,c_1),d(v,c_2)) = i.$ There are $2i$ vertices $w$ in $P_n$ with \\
\noindent $\min(d(w,c_1),d(w,c_2)) < i$, so $\peri(v) = 2i$.

Now suppose that $n$ is odd, $c$ is the center vertex, and $d(v,c) = i$. If $i = 0$, then $v = c$ and $\peri(v) = 0$, since there is no vertex $w \in P_n$ with $n_{P_n}(w, c) > n_{P_n}(c, w)$. If $i \ge 1$, then $\peri(v) = 2i-1$ since $n_{P_n}(c, v) > n_{P_n}(v, c)$ and $n_{P_n}(w, v) > n_{P_n}(v, w)$ for all vertices $w \in V(P_n)$ with $d(w, c) < i$.
\end{proof}

In the next two results, we determine the peripherality of balanced spiders and of vertices in balanced spiders.

\begin{thm}
If $n = 1+a b$, then $\peri(S_{a,b}) = \binom{n}{2}-b\binom{a}{2}$
\end{thm}

\begin{proof}
To compute $\peri(S_{a,b})$, we subtract $\sum_{\{u,v\}\subset V(S_{a,b})}\mathbbm{1}{\left[n_{S_{a,b}}(u,v) = n_{S_{a,b}}(v,u)\right]}$ from $\frac{1}{2}n(n-1)$. Let the vertices of $S_{a,b}$ consist of the center vertex $c$ and the leg vertices $v_{i,j}$ with $1 \le i \le a$ and $1 \le j \le b$ where the vertices $v_{i, 1}, v_{i, 2}, \dots, v_{i,b}$ are on the same leg for each $i$ and $d(c, v_{i, j}) = j$ for all $i, j$.

The only way that we have $n_{S_{a,b}}(u,v) = n_{S_{a,b}}(v,u)$ is if $u = v_{i,j}$ and $v = v_{i',j}$ for some $1 \le i, i' \le a$ and $1 \le j \le b$. The number of subsets $\{v_{i,j},v_{i',j}\}\subset V(S_{a,b})$ is $\binom{a}{2}b$.
\end{proof}

In the last result, we determined $\peri(S_{a,b})$ without computing the values of $\peri(v)$ for the vertices $v$ in $S_{a,b}$. We determine the peripherality of the vertices of $S_{a,b}$ in the next result.

\begin{thm}
If $v$ has distance $j \ge 1$ from the center in $S_{a,b}$, then $\peri(v) = 1+a(j-1)$. If $v$ is the center, then $\peri(v)= 0$.
\end{thm}

\begin{proof}
Note that for $u \in V(S_{a,b})-v$, we have $n_G(u,v) > n_G(v,u)$ if and only if $u$ is closer to the center of $S_{a,b}$ than $v$. Including the center, there are $1+a(j-1)$ vertices $u \in V(G)-v$ that have distance at most $j-1$ to the center.
\end{proof}

We determine the maximum possible value of $\peri(T)$ among all trees $T$ of order $n$ in the next result for all $n \ge 9$. As a corollary, we also obtain the maximum possible value of $\peri(G)$ among all graphs $G$ of order $n$.

\begin{thm}\label{maxperitree}
For all $n \ge 9$, the maximum possible value of $\peri(T)$ among all trees $T$ of order $n$ is $\binom{n}{2}$.
\end{thm}

\begin{proof}
For all graphs $G$ of order $n$, we have $\peri(G) \le \binom{n}{2}$ by definition. In order to prove the lower bound, for any $n \ge 10$, we will construct an unbalanced spider $S$ of order $n$ with three legs of lengths $a, b, c$ such that $a < b < c$ and $1+a+b > c$. Specifically if $n = 3k+1$ for some integer $k$, then we let $(a,b,c) = (k-1, k, k+1)$. If $n = 3k+2$, then we let $(a, b, c) = (k-1, k, k+2)$. If $n = 3k+3$, then we let $(a, b, c) = (k-1,k+1, k+2)$.

Let $u$ and $v$ be two distinct vertices in $S$. If $u$ is the degree-$3$ vertex of $S$, then $n_S(u, v) > n_S(v, u)$. Thus we may assume that $u$ and $v$ are both leg vertices of $S$. If $u$ and $v$ are on the same leg and $u$ is closer to the degree-$3$ vertex of $S$, then $n_S(u, v) > n_S(v, u)$. So now we may assume that $u$ and $v$ are on different legs. If $u$ is closer to the degree-$3$ vertex than $v$, then the vertices on the leg containing $u$ and the leg that contains neither $u$ nor $v$ are all closer to $u$ than to $v$. Thus in this case we also have $n_S(u, v) > n_S(v, u)$, since $1+a+b > c$.

The final case is when $u$ and $v$ are on different legs with the same distance to the degree-$3$ vertex. Without loss of generality, suppose that the leg containing $u$ is longer than the leg containing $v$. Then again we have $n_S(u, v) > n_S(v, u)$. Thus $\peri(S) = \binom{n}{2}$.
\end{proof}

\begin{cor}
For all $n \ge 9$, the maximum possible value of $\peri(G)$ among all graphs $G$ of order $n$ is $\binom{n}{2}$.
\end{cor}

In the next two results, we determine the peripherality of complete bipartite graphs and vertices in complete bipartite graphs.

\begin{thm}\label{perikmnv}
In the complete bipartite graph $K_{m,n}$ with $m < n$, the vertices $u$ on the left have $\peri(u) = 0$ and the vertices $v$ on the right have $\peri(v) = m$. If $m = n$, then all vertices $w$ on both sides have $\peri(w) = 0$.
\end{thm}

\begin{proof}
Suppose that $m < n$. If $u$ is any vertex on the left and $v$ is any vertex on the right, then $n_{K_{m,n}}(v,u) = m < n = n_{K_{m,n}}(u, v)$. Moreover $n_{K_{m,n}}(x_1,x_2) = 0$ for any two vertices in the same part of $K_{m,n}$. Then $\peri(u) = 0$ for any vertex $u$ on the left, and $\peri(v) = m$ for any vertex $v$ on the right.

Now suppose that $m = n$. If $u$ is any vertex on the left and $v$ is any vertex on the right, then $n_{K_{m,n}}(v,u) = m = n = n_{K_{m,n}}(u, v)$. Thus $\peri(w) = 0$ for all vertices $w$ in both parts.
\end{proof}

We multiply the last result by $n$ (which is at least $m$) to obtain $\peri(K_{m,n})$.

\begin{cor}
If $m < n$, then $\peri(K_{m,n}) = n m$. If $m = n$, then $\peri(K_{m,n}) = 0$.
\end{cor}

\begin{proof}
Let $m < n$. By Theorem \ref{perikmnv}, $K_{m,n}$ has $n$ vertices $v$ with $\peri(v) = m$ and all other vertices $u$ have $\peri(u) = 0$. Thus $\peri(K_{m,n}) = n m$. If $m = n$, then $\peri(w) = 0$ for all vertices $w$ in both parts, so $\peri(K_{m,n}) = 0$.
\end{proof}

It is simple to obtain a sharp bound on the sum of the peripherality and degree of a single vertex with respect to the order of a graph.

\begin{thm}
For any vertex $v$ in a graph of order $n$, the maximum possible value of $\peri(v)+\deg(v)$ is $2n-4$.
\end{thm}

\begin{proof}
For the upper bound, note that $\deg(v) \le n-1$, and $\peri(v) = 0$ if $\deg(v) = n-1$. Moreover $\peri(v) \le  n-2$ if $\deg(v) = n-2$, since the vertex $u$ not adjacent to $v$ cannot have $n_G(u,v) > n_G(v,u)$. Thus $\peri(v)+\deg(v) \le 2n-4$.

For the lower bound, consider the graph $G$ obtained from $K_{n-1}$ by adding a new vertex $v$ of degree $n-2$. Then $\peri(v) = n-2$, so $\peri(v)+\deg(v) = 2n-4$.
\end{proof}

We also find an analogous sharp upper bound on the sum of peripherality and degree when we restrict to trees.

\begin{thm}
For any vertex $v$ in a tree of order $n \ge 5$, the maximum possible value of $\peri(v)+\deg(v)$ is $n$.
\end{thm}

\begin{proof}
For the upper bound, let $v$ be any vertex in a tree $T$ of order $n$, and let the neighbors of $v$ in $T$ be $u_1, \dots, u_k$, so that $\deg(v) = k$. Then there are $k$ connected subgraphs $S_i$ obtained from removing $v$, where $S_i$ contains the vertex $u_i$ for each $i$.

Since $T$ is a tree, each $S_i$ contains a leaf $\ell_i$, which may be equal to $u_i$. For every $j \neq i$ and every vertex $w$ in $S_j$, we have $d(v,w) < d(\ell_i,w)$. Thus $n_T(v, \ell_i) \ge 1+\sum_{j \neq i} |S_j| = n - |S_i|$, so there is at most one value of $i$ for which $n_T(\ell_i,v) \ge n_T(v, \ell_i)$. So there are at least $k-1$ values of $i$ for which $n_T(\ell_i,v) < n_T(v, \ell_i)$.

Thus $\peri(v) \le (n-1)-(k-1) = n-\deg(v)$, so we have $\peri(v)+\deg(v) \le n$.

For the lower bound, consider the unbalanced spider obtained from the star $K_{1,n-2}$ by adding a new vertex $v$ with a single new edge between $v$ and one of the leaves of $K_{1,n-2}$. Then $\peri(v) = n-1$ and $\deg(v) = 1$, so $\peri(v)+\deg(v) = n$.
\end{proof}

\section{Extremal bounds and exact values for $\eperi(G)$}\label{eperires}

In this section, we determine the exact values of the edge peripherality of cycles, complete graphs, complete bipartite graphs, paths, and balanced spiders. We also determine the edge peripherality of the individual edges in these graphs. Moreover we determine the maximum possible value of $\eperi(e)+\edeg(e)$ in any graph of order $n$, as well as any tree of order $n$. Furthermore we asymptotically determine the maximum possible value of $\eperi(T)$ among all trees $T$ of order $n$.

As with $\espr(G)$ and $\peri(G)$, it is possible to find graphs $G$ for which $\eperi(G) = 0$.

\begin{thm}
The minimum possible value of $\eperi(G)$ among all graphs $G$ of order $n$ is $0$.
\end{thm}

\begin{proof}
As with $\peri(G)$, any vertex-transitive graph $G$ has $\eperi(G) = 0$. More generally, if $G$ is the disjoint union of (possibly different) vertex-transitive graphs $G_1, \dots, G_k$ each of the same order, then $\eperi(G) = 0$.
\end{proof}

The graphs $C_n$ and $K_n$ are vertex-transitive, and it is clear that $\eperi(e) = 0$ for all edges in $K_{m,n}$. Thus we also have the next theorem and corollary.

\begin{thm}
$\eperi(C_n) = \eperi(K_n) = \eperi(K_{m,n}) = 0$ for all $m \ge 1$, $n \ge 1$.
\end{thm}

\begin{cor}
The minimum possible value of $\eperi(T)$ among all trees $T$ of order $n$ is $0$.
\end{cor}

In the next two results, we determine the edge peripherality of paths and the edges on a path. 

\begin{thm}
For $n \ge 3$, let the vertices of $P_n$ be $v_1, \dots, v_n$ in path order. If $n$ is even and $c_1, c_2$ are the center vertices, then $\eperi(e) = 2i$ for any edge $e = \left\{u, v\right\}$ with \\
\noindent $\min(d(u,c_1),d(u,c_2),d(v,c_1),d(v,c_2)) = i$. If $n$ is odd and $c$ is the center vertex, then $\eperi(e) = 2i-1$ for any edge $e = \left\{u, v\right\}$ with $\min(d(u,c),d(v,c)) = i \ge 1$, and $\eperi(e) = 0$ for any edge $e$ incident to $c$.
\end{thm}

\begin{proof}
Suppose that $n$ is even and $c_1, c_2$ are the center vertices. Let $e = \left\{u, v\right\}$ with \\
\noindent $\min(d(u,c_1),d(u,c_2),d(v,c_1),d(v,c_2)) = i.$ There are $2i$ vertices $w$ in $P_n$ with \\
\noindent $\min(d(w,c_1),d(w,c_2)) < i$, so $\eperi(e) = 2i$.

Now suppose that $n$ is odd and $c$ is the center vertex. Let $e = \left\{u, v\right\}$ with $\min(d(u,c),d(v,c)) = i$. If $i = 0$, then $\eperi(e) = 0$, since there is no vertex $w \in P_n$ with $n_{P_n}(w, c) > n_{P_n}(c, w)$. If $i \ge 1$, then $\eperi(e) = 2i-1$ since $n_{P_n}(c, u) > n_{P_n}(u, c)$, $n_{P_n}(c, v) > n_{P_n}(v, c)$, $n_{P_n}(w, u) > n_{P_n}(u, w)$, and $n_{P_n}(w, v) > n_{P_n}(v, w)$ for all vertices $w \in V(P_n)$ with $d(w, c) < i$.
\end{proof}

By summing the values from the last theorem, we obtain $\eperi(P_n)$.

\begin{thm}
For $n \ge 3$, we have $\eperi(P_n) = \frac{1}{2}(n-2)(n-4)$ if $n$ is even and $\eperi(P_n) = \frac{1}{2}(n-3)^2$ if $n$ is odd.
\end{thm}

\begin{proof}
Let the vertices of $P_n$ be $v_1, \dots, v_n$ in order. If $n$ is even, then $\eperi(P_n) = 2 \sum_{i = 0}^{\frac{n}{2}-2} 2i = 2(\frac{n}{2}-2)(\frac{n}{2}-1)$. If $n$ is odd, then $\eperi(P_n) = 2 \sum_{i = 0}^{\frac{n-1}{2}-2} (2i+1) = 2(\frac{n-1}{2}-1)^2$.
\end{proof}

In the next two results, we determine the edge peripherality of balanced spiders and the edges on a balanced spider. 

\begin{thm}
For edges $e$ of $S_{a,b}$ with vertices that are distances $x+1$ and $x+2$ from the center of the spider with $x \ge 0$, we have $\eperi(e) = 1+a x$. For edges $e$ incident to the center vertex of $S_{a, b}$, we have $\eperi(e) = 0$.
\end{thm}

\begin{proof}
If $e$ is incident to the center vertex $c$, then $\eperi(e) = 0$ since there is no vertex $u \in V(S_{a, b})$ with $n_{S_{a,b}}(u, c) > n_{S_{a,b}}(c, u)$. If $e$ contains vertices $v$ and $w$ that are distances $x+1$ and $x+2$ respectively from the center of the spider with $x \ge 0$, we have $\eperi(e) = 1 + a x$ since $n_{S_{a,b}}(c,v) > n_{S_{a,b}}(v,c)$, $n_{S_{a,b}}(c,w) > n_{S_{a,b}}(w,c)$, $n_{S_{a,b}}(y,v) > n_{S_{a,b}}(v,y)$, and $n_{S_{a,b}}(y,w) > n_{S_{a,b}}(w,y)$ for all vertices $y$ with distance at most $x$ to $c$. 
\end{proof}

By summing the values from the last theorem, we obtain $\eperi(S_{a,b})$.

\begin{thm}
If $n = 1+a b$, then $\eperi(S_{a,b}) = n - 1 - a + a^2 (\frac{(b-1)(b-2)}{2})$.
\end{thm}

\begin{proof}
For an edge $e$ with vertices that are distances $x+1$ and $x+2$ from the center of the spider with $x \ge 0$, we have $\eperi(e) = 1+a x$. Thus $\eperi(S_{a,b}) = a \sum_{x = 0}^{b-2} (1+a x) = a(b-1 + a \frac{(b-1)(b-2)}{2}) = n - 1 - a + a^2 (\frac{(b-1)(b-2)}{2})$.
\end{proof}

Next we asymptotically determine the maximum possible value of $\eperi(T)$ among all trees $T$ of order $n$. 

\begin{thm}
The maximum possible value of $\eperi(T)$ among all trees $T$ of order $n$ is $\frac{n^2}{2} -\Theta(n)$.
\end{thm}

\begin{proof}
The lower bound follows from $P_n$. For the upper bound, let $x_1, \dots, x_{n}$ be an ordering of the vertices in $T$ such that the subgraph of $T$ on $x_1, \dots, x_i$ is a tree for each $i = 1, \dots, n$. Let $e_1, \dots, e_{n-1}$ be the edges of $T$ such that $e_i$ contains $x_{i+1}$ and a vertex in $\left\{x_t : t \le i \right\}$ for each $1 \le i \le n-1$.

Note that $\eperi(T) = \sum_{i = 1}^{n-1} \eperi(e_i)$. Moreover $\eperi(e_i) \le \peri(x_{i+1})$ for each $1 \le i \le n-1$ and $\peri(T) = \sum_{i = 1}^{n} \peri(x_i)$, so $\eperi(T) \le \peri(T)$. Thus the upper bound follows from Theorem \ref{maxperitree}.
\end{proof}

As with peripherality and standard vertex degree, we obtain a sharp bound on the sum of the edge degree and edge peripherality with respect to the order of a graph.

\begin{thm}
For any edge $e$ in a graph of order $n \ge 5$, the maximum possible value of $\eperi(e)+\edeg(e)$ is $2n-4$.
\end{thm}

\begin{proof}
The upper bound is immediate from the definitions of $\eperi(e)$ and $\edeg(e)$, since $\eperi(e) \le n-2$ and $\edeg(e) \le n-2$.

For the lower bound, consider the graph $G$ obtained from $K_{n-2}$ by adding two new vertices $v$ and $w$ of degrees $\lceil \frac{n-2}{2} \rceil+1$ and $\lfloor \frac{n-2}{2} \rfloor+1$ respectively with disjoint neighborhoods so that $\left\{u,v\right\}$ is an edge. Then $\edeg(\left\{u,v\right\}) = n-2$ and $\eperi(\left\{u,v\right\}) = n-2$, so $\eperi(\left\{u,v\right\})+\edeg(\left\{u,v\right\}) = 2n-4$.
\end{proof}

We also find an analogous sharp upper bound on the sum of edge peripherality and edge degree when we restrict to trees.

\begin{thm}
For any edge $e$ in a tree of order $n \ge 6$, the maximum possible value of $\eperi(e)+\edeg(e)$ is $n-1$.
\end{thm}

\begin{proof}
For the upper bound, let $e = \left\{u,v\right\}$ be any edge in a tree $T$ of order $n$, and let $w_1, \dots, w_k$ be the vertices that are neighbors of $u$ or $v$ in $T$, so that $\edeg(e) = k$. Then there are $k$ connected subgraphs $S_i$ obtained from removing both $u$ and $v$, where $S_i$ contains the vertex $w_i$ for each $i$. Since $T$ is a tree, each $S_i$ contains a leaf $\ell_i$, which may be equal to $w_i$. 

For the first case, suppose that $\ell_i$ is closer to $u$ than to $v$. For every $j \neq i$ and every vertex $w$ in $S_j$, we have $d(u,w) < d(\ell_i,w)$. Thus $n_T(u, \ell_i) \ge 2+\sum_{j \neq i} |S_j| = n - |S_i|$, so there is at most one value of $i$ with $\ell_i$ closer to $u$ than to $v$ for which $n_T(\ell_i,u) \ge n_T(u, \ell_i)$.

For the second case, suppose that $\ell_i$ is closer to $v$ than to $u$. Using the same argument as the first case, we have $n_T(v, \ell_i) \ge 2+\sum_{j \neq i} |S_j| = n - |S_i|$, so there is at most one value of $i$ with $\ell_i$ closer to $v$ than to $u$ for which $n_T(\ell_i,v) \ge n_T(v, \ell_i)$.

Now suppose that there exists $i$ with $\ell_i$ closer to $u$ than to $v$ for which $n_T(\ell_i,u) \ge n_T(u, \ell_i)$. Then $|S_i| \ge \frac{n}{2}$, so $n_T(u, \ell_j) \ge \frac{n}{2}+1 > n_T(\ell_j, u)$ for all $j \neq i$ (this includes both $j$ for which $\ell_j$ is closer to $v$ than to $u$, as well as $j$ for which $\ell_j$ is closer to $u$ than to $v$). 

Similarly, if there exists $i$ with $\ell_i$ closer to $v$ than to $u$ for which $n_T(\ell_i,v) \ge n_T(v, \ell_i)$, then $n_T(v, \ell_j) > n_T(\ell_j, v)$ for all $j \neq i$, including both $j$ for which $\ell_j$ is closer to $v$ than to $u$, and $j$ for which $\ell_j$ is closer to $u$ than to $v$. Thus there is at most one value of $i$ for which $n_T(\ell_i,u) \ge n_T(u, \ell_i)$ and $n_T(\ell_i,v) \ge n_T(v, \ell_i)$.

Hence there are at least $k-1$ values of $i$ for which $n_T(\ell_i,u) < n_T(u, \ell_i)$ or $n_T(\ell_i,v) < n_T(v, \ell_i)$, so we have $\eperi(e) \le (n-2)-(k-1) = n-1-k$. Therefore $\eperi(e)+\edeg(e) \le n-1$.

For the lower bound, consider the unbalanced spider obtained from the star $K_{1,n-3}$ by adding two new vertices $u$ and $v$ with an edge $e$ between $u$ and $v$, as well as an edge between $u$ and one of the leaves of $K_{1,n-3}$. Then $\eperi(e) = n-2$ and $\edeg(e) = 1$, so $\eperi(e)+\edeg(e) = n-1$. 
\end{proof}

Finally, we determine the maximum possible value of $\eperi(G)$ up to a constant factor among all graphs $G$ of order $n$.

\begin{thm}
The maximum possible value of $\eperi(G)$ among all graphs $G$ of order $n$ is $\Theta(n^3)$.
\end{thm}

\begin{proof}
The lower bound follows from the following construction. Let $H$ be a graph of order $5n$ with vertices in five parts $V_1, V_2, V_3, V_4, V_5$ each of size $n$ such that there are $n^2$ edges between the $n$ vertices of $V_i$ and the $n$ vertices of $V_{i+1}$ for each $i = 1, 2, 3, 4$, so that $H$ has $4n^2$ total edges. Note that for any edge $e$ between a vertex $u \in V_1$ and a vertex $v \in V_2$, any vertex $x \in V_3$ will have $n_H(x, u) > n_H(u, x)$ and $n_H(x, v) > n_H(v, x)$. The same is true for edges between vertices of $V_4$ and $V_5$. Thus $\eperi(H) \ge 2n^3$. Note that we can modify the construction so that it also works for graph orders not divisible by $5$ by allowing the number of vertices in each part to differ by at most $1$.

For the upper bound, observe that $\eperi(G) < \frac{1}{2}n^3$ since $G$ has at most $\binom{n}{2}$ edges and for each edge $e$ there are $n-2$ vertices $x$ not in $e$.
\end{proof}

\section{NP-completeness of problems about the Mostar index, irregularity, peripherality, and eccentricity in cliques}\label{npcliqueproblems}

In this section, we discuss several problems about cliques which involve Mostar index, irregularity, vertex peripherality, edge peripherality, vertex eccentricity, and edge eccentricity. We consider problems with equality constraints, as well as problems with inequality constraints.

\subsection{Clique problems with equality constraints}

In this subsection, we consider several problems about cliques with equality constraints. We show that each problem is NP-complete. 

Let $CLIQUE(G,k)$ be the problem of determining whether $G$ has a clique of size $k$. This problem is known to be NP-complete in general \cite{npclique}.

Let $CLIQUE_{Mostar,=}(G,k)$ be the problem of determining whether or not $G$ has a clique on $k$ vertices for which every edge in the clique has the same Mostar index. We show that this problem is NP-complete.

\begin{thm}\label{mostarnpclique}
The problem $CLIQUE_{Mostar,=}(G,k)$ is NP-complete.
\end{thm}

\begin{proof}
To see that the problem is in NP, suppose that $G$ has a clique of size $k$ with all edges having the same Mostar index. Then a prover can provide the vertices of the clique with edges of the same Mostar index, and a verifier can check that the vertices form a clique, and if so, then calculate the Mostar indices of the edges in the clique. 

To see that the problem is NP-hard, we reduce $CLIQUE(G,k)$ to $CLIQUE_{Mostar,=}(H,k)$ for a graph $H$ that we construct from $G$ for each $k \ge 4$. Suppose that $G$ is any graph of order $n$, and let $k \ge 4$.

Let $G’$ be the graph obtained from $G$ by adding pendent vertices to the vertices of $G$ until all vertices of $G$ have the same degree in $G’$. Let $H$ be obtained from $G’$ by adding a universal vertex $c$ along with $2|V(G’)|$ pendent vertices whose only neighbor in $H$ is $c$.

Note that all pairs of vertices in $H$ have distance $1$ or $2$, and the only pairs of vertices of distance $1$ in $H$ are pairs that had an edge in $G$ or pairs that contain $c$ or a pendent vertex from $G’$. For all vertices $u,v \in V(G)$, $u$ and $v$ have the same degree in $G’$, so $n_H(u,v)-n_H(v,u) = 0$. 

Observe that if $v$ is a pendent vertex in $G’$ that was not present in $G$, then $v$ has degree $2$ in $H$. So $v$ cannot be present in any clique of size at least $4$ in $H$. Moreover if $u$ is a pendent vertex that is present in $H$ but not $G’$, then $u$ cannot be present in any clique of size at least $3$. 

Thus the only vertices in $H$ that can be present in a clique of size at least $4$ in $H$ are $c$ and the vertices of $G$. However, every edge that contains $c$ has greater Mostar index than any edge which does not contain $c$, because of the $2|V(G’)|$ pendent vertices in $H$. Thus no clique of size at least $3$ in $H$ with all edges of the same Mostar index can contain $c$.

Therefore $H$ has a clique of size $k$ with all edges of the same Mostar index if and only if $G$ has a clique of size $k$. This completes the reduction from $CLIQUE(G,k)$ to $CLIQUE_{Mostar,=}(H,k)$.
\end{proof}

Next we turn to irregularity. Let $CLIQUE_{irr, =}(G,k)$ be the problem of determining whether or not $G$ has a clique on $k$ vertices for which every edge in the clique has the same irregularity. We show that this problem is NP-complete.

\begin{thm}
The problem $CLIQUE_{irr, =}(G,k)$ is NP-complete.
\end{thm}

\begin{proof}
The problem is clearly in NP. To see why the problem is NP-hard, we can use the same construction of $H$ as in Theorem \ref{mostarnpclique}. 

Since every pair of vertices in $H$ have distance at most $2$, we have $n_H(u,v)-n_H(v,u) = d_H(u)-d_H(v)$ for all vertices $u,v \in V(H)$, where $d_H(u)$ denotes the degree of $u$ in $H$. Thus $CLIQUE(G,k)$ reduces to $CLIQUE_{irr,=}(H,k)$.
\end{proof}

Now we turn to vertex peripherality. Let $CLIQUE_{peri, =}(G,k)$ be the problem of determining whether or not $G$ has a clique on $k$ vertices for which every vertex in the clique has the same peripherality. We show that this problem is NP-complete.

\begin{thm}\label{nppericlique}
The problem $CLIQUE_{peri, =}(G,k)$ is NP-complete.
\end{thm}

\begin{proof}
The problem is clearly in NP. To see why the problem is NP-hard, let $G$ be any graph and $k \ge 4$. We use the same construction of the graph $H$ as in Theorem \ref{mostarnpclique}, with the only difference that we remove any connected components of size at most $2$ from $G$ before we construct $G’$. If $G$ is empty after removing the connected components of size at most $2$, then $H$ only has the vertex $c$, so $H$ has no clique of size $k$. Therefore suppose for the rest of the proof that $G$ has a connected component of size at least $3$, so there is a vertex of degree at least $2$ in $G$. In the proof of Theorem \ref{mostarnpclique}, we showed that $n_H(u,v)-n_H(v,u) = 0$ for all vertices $u,v \in V(G)$. 

Moreover we have $n_H(c, u) > n_H(u, c)$ for all vertices $u \in V(H)$ with $u \neq c$, since $c$ is the only universal vertex in $H$. If $p$ is a pendent vertex in $H$, then we have $n_H(p, v) < n_H(v, p)$ for any non-pendent vertex $v$ in $H$. This is because $p$ and $v$ are both adjacent to $c$ in $H$, $p$ has no other neighbors in $H$ and $v$ has at least one other neighbor, and any pair of vertices in $H$ have distance at most $2$. If $q$ is a pendent vertex in $G’$, then $q$ has degree $2$ in $H$, and $n_H(q, u) < n_H(u,q)$ for any vertex $u \in V(G)$ that was not in a connected component of size at most $2$ in $G$. This is because $d_H(u)-d_H(v) = n_H(u,v)-n_H(v,u)$ for all vertices $u, v \in V(H)$, and $d_H(q) = 2 < d_H(u)$ since all vertices from $V(G)$ in $H$ have the same degree in $H$, and there is at least one vertex from $V(G)$ of degree at least $3$ in $H$ since we assumed that there existed a vertex of degree at least $2$ in $G$ and we added a universal vertex to form $H$.

Let $\peri_H(v)$ denote the peripherality of vertex $v$ in $H$. Then $\peri_H(v) = 1$ for all vertices $v$ in $V(G)$ and $\peri(c) = 0$. As in Theorem \ref{mostarnpclique}, note that any clique of size $k$ in $H$ cannot contain any pendent vertices of $H$ since they have degree $1$ in $H$. Moreover any clique of size $k$ in $H$ cannot contain any vertices that were pendent vertices in $G’$, since these vertices have degree $2$ in $H$. Thus a clique of size $k$ in $H$ can only contain $c$ and vertices from $G$, so a clique of size $k$ in $H$ with all vertices having the same peripherality can only contain vertices from $G$. Thus $H$ has a clique of size $k$ with all vertices of the same peripherality if and only if $G$ has a clique of size $k$. Thus $CLIQUE(G,k)$ reduces to $CLIQUE_{peri,=}(H,k)$.
\end{proof}

Next we turn to edge peripherality. Let $CLIQUE_{eperi, =}(G,k)$ be the problem of determining whether or not $G$ has a clique on $k$ vertices for which every edge in the clique has the same edge peripherality. We show that this problem is NP-complete.

\begin{thm}
The problem $CLIQUE_{eperi, =}(G,k)$ is NP-complete.
\end{thm}

\begin{proof}
The problem is clearly in NP. To see why the problem is NP-hard, let $G$ be any graph and $k \ge 4$. We use the same construction of the graph $H$ as in Theorem \ref{nppericlique}, where we remove any connected components of size at most $2$ from $G$ before we construct $G’$. Again, if $G$ is empty after removing the connected components of size at most $2$, then $H$ only has the vertex $c$, so $H$ has no clique of size $k$. Therefore suppose for the rest of the proof that $G$ has a connected component of size at least $3$, so there is a vertex of degree at least $2$ in $G$. In the proof of Theorem \ref{mostarnpclique}, we showed that $n_H(u,v)-n_H(v,u) = 0$ for all vertices $u,v \in V(G)$. As in the proof of Theorem \ref{nppericlique}, we have $n_H(c, u) > n_H(u, c)$ for all vertices $u \in V(H)$ with $u \neq c$, since $c$ is the only universal vertex in $H$. If $p$ is a pendent vertex in $H$, then we have $n_H(p, v) < n_H(v, p)$ for any non-pendent vertex $v$ in $H$. If $q$ is a pendent vertex in $G’$, then $q$ has degree $2$ in $H$, and $n_H(q, u) < n_H(u,q)$ for any vertex $u \in V(G)$ that was not in a connected component of size at most $2$ in $G$. 

Let $\eperi_H(e)$ denote the edge peripherality of edge $e$ in $H$. Then $\eperi_H(e) = 0$ for any edges $e$ in $H$ containing $c$, and $\eperi_H(e) = 1$ for any edges $e$ containing two vertices from $V(G)$. As in Theorems \ref{mostarnpclique} and \ref{nppericlique}, note that any clique of size $k$ in $H$ cannot contain any pendent vertices of $H$ since they have degree $1$ in $H$. Moreover any clique of size $k$ in $H$ cannot contain any vertices that were pendent vertices in $G’$, since these vertices have degree $2$ in $H$. Thus a clique of size $k$ in $H$ can only contain $c$ and vertices from $G$, so a clique of size $k$ in $H$ with all edges having the same edge peripherality can only contain vertices from $G$. Thus $H$ has a clique of size $k$ with all edges of the same edge peripherality if and only if $G$ has a clique of size $k$. Thus $CLIQUE(G,k)$ reduces to $CLIQUE_{eperi,=}(H,k)$.
\end{proof}

Now we turn to vertex eccentricity. Let $CLIQUE_{ecc, =}(G,k)$ be the problem of determining whether or not $G$ has a clique on $k$ vertices for which every vertex in the clique has the same eccentricity. We show that this problem is NP-complete.

\begin{thm}\label{npeccclique}
The problem $CLIQUE_{ecc, =}(G,k)$ is NP-complete.
\end{thm}

\begin{proof}
The problem is clearly in NP. To see why the problem is NP-hard, let $G$ be any graph and $k \ge 2$. Let $X$ be the graph obtained from $G$ by adding a universal vertex $c$ and a pendent vertex $p$ whose only neighbor is $c$. Then $c$ is the only vertex with eccentricity $1$ in $X$, all other vertices have eccentricity $2$. Thus $X$ has a clique of size $k$ with all vertices of the same eccentricity if and only if $G$ has a clique of size $k$. Thus $CLIQUE(G,k)$ reduces to $CLIQUE_{ecc,=}(X,k)$.
\end{proof}

Finally we turn to edge eccentricity. Let $CLIQUE_{eecc, =}(G,k)$ be the problem of determining whether or not $G$ has a clique on $k$ vertices for which every edge in the clique has the same edge eccentricity. We show that this problem is NP-complete.

\begin{thm}
The problem $CLIQUE_{eecc, =}(G,k)$ is NP-complete.
\end{thm}

\begin{proof}
The problem is clearly in NP. To see why the problem is NP-hard, let $G$ be any graph and $k \ge 3$. As in Theorem \ref{npeccclique}, let $X$ be the graph obtained from $G$ by adding a universal vertex $c$ and a pendent vertex $p$ whose only neighbor is $c$. Then every edge containing $c$ has edge eccentricity $1$ in $X$, and all other edges have eccentricity $2$, so $X$ does not have a clique of size $k$ containing $c$ with all edges of the same edge eccentricity. Thus $X$ has a clique of size $k$ with all edges of the same edge eccentricity if and only if $G$ has a clique of size $k$. Thus $CLIQUE(G,k)$ reduces to $CLIQUE_{eecc,=}(X,k)$.
\end{proof}

\subsection{Clique problems with inequality constraints}

Now we consider several clique problems with inequality constraints. Unlike the last subsection, some of these problems are in P.

Let $CLIQUE_{Mostar,\neq}(G,k)$ be the problem of determining whether or not $G$ has a clique on $k$ vertices for which every edge in the clique has a different Mostar index. We show that this problem is NP-complete.

\begin{thm}\label{mostarnpcliqued}
The problem $CLIQUE_{Mostar, \neq}(G,k)$ is NP-complete.
\end{thm}

\begin{proof}
The problem is clearly in NP. To show that it is NP-complete, our proof is very similar to Theorem \ref{mostarnpclique}, except we modify the construction and we reduce $CLIQUE(G,k)$ to $CLIQUE_{Mostar, \neq}(J,k+1)$, where $J$ is a graph that we construct based on $G$. Let $k \ge 3$. If $G$ has order $n \ge 2$, then for each vertex $v_1, v_2, \dots, v_n \in V(G)$, we add $4^{n+i}$ pendent vertices to $v_i$ to form $G_1$ from $G$. 

We construct a graph $J$ from $G_1$ by adding a universal vertex $c$ with $4^{4n}$ pendent vertices whose only neighbor is $c$. First we claim that $|n_J(u,v)-n_J(v,u)|$ is unique for every subset of vertices $\left\{u,v\right\} \subset V(G)$. 

To see why this is true, note that we have $4^{n+j}-4^{n+i}-n < n_J(v_j,v_i)-n_J(v_i,v_j) < 4^{n+j}-4^{n+i}+n$ for all $i < j$. Thus we have $4^{n+j-1} < n_J(v_j,v_i)-n_J(v_i,v_j) < 4^{n+j}$, so the value of $j$ is determined by $n_J(v_j,v_i)-n_J(v_i,v_j)$ if $i < j$. Observe that if $i < i’$, then we have $(4^{n+j}-4^{n+i})-(4^{n+j}-4^{n+i’}) > 2n$, so the ordered pair $(i,j)$ is determined by $n_J(v_j,v_i)-n_J(v_i,v_j)$ for all $i, j \le n$.

Note that every edge in $J$ that contains $c$ and a vertex from $G$ must have a different Mostar index, since $c$ is adjacent to all other vertices in $J$, and all vertices from $G$ have different degree in $J$. Any clique of size $k+1$ in $J$ cannot contain any pendent vertices of $J$ or vertices that were pendent vertices in $G_1$, in both cases since their degree is too low. Thus any clique of size $k+1$ in $J$ can only contain $c$ or vertices of $G$. 

Any clique of size $k+1$ in $J$ has all edges with different Mostar indices. This is because any pair of edges in the clique that both do not contain $c$ have different Mostar indices, any pair of edges in the clique that both contain $c$ have different Mostar indices, all edges containing $c$ have Mostar index greater than $4^{3n}$, and all edges that do not contain $c$ have Mostar index less than $4^{3n}$.

Thus $J$ has a clique of size $k+1$ with all edges having different Mostar indices if and only if $G$ has a clique of size $k$. This completes the reduction from $CLIQUE(G,k)$ to $CLIQUE_{Mostar, \neq}(J,k+1)$.
\end{proof}

Let $CLIQUE_{irr, \neq}(G,k)$ be the problem of determining whether or not $G$ has a clique on $k$ vertices for which every edge in the clique has a different irregularity. We show that this problem is NP-complete.

\begin{thm}
The problem $CLIQUE_{irr, \neq}(G,k)$ is NP-complete.
\end{thm}

\begin{proof}
The problem is clearly in NP. To see why the problem is NP-hard, we can use the same construction of $J$ as in Theorem \ref{mostarnpcliqued} since $n_J(u,v)-n_J(v,u) = d_J(u)-d_J(v)$ for all vertices $u,v \in V(J)$.
\end{proof}

Let $CLIQUE_{peri, \neq}(G,k)$ be the problem of determining whether or not $G$ has a clique on $k$ vertices for which every vertex in the clique has a different peripherality. We show that this problem is NP-complete.

\begin{thm}\label{perineq}
The problem $CLIQUE_{peri, \neq}(G,k)$ is NP-complete.
\end{thm}

\begin{proof}
The problem is clearly in NP. To see why the problem is NP-hard, let $k \ge 3$. We can use the same construction of $J$ as in Theorem \ref{mostarnpcliqued}, with the only difference that we remove any connected components of size at most $2$ from $G$ before we construct $G_1$. If $G$ is empty after removing the connected components of size at most $2$, then $J$ only has the vertex $c$ and a pendent vertex, so $J$ has no clique of size $k+1$. Therefore suppose that $G$ has a connected component of size at least $3$, so there is a vertex of degree at least $2$ in $G$. 

Recall that $n_J(u,v)-n_J(v,u) = d_J(u)-d_J(v)$ for all vertices $u,v \in V(J)$. If $v_1, \dots, v_n$ are the vertices of $G$, then $n_J(v_x, v_y) < n_J(v_y,v_x)$ for all $x < y$ by construction. Moreover we have $n_J(v_i, c) < n_J(c, v_i)$ for all $i = 1, \dots, n$. 

If $p$ is a pendent vertex in $J$, then $n_J(p, u) < n_J(u, p)$ for all vertices $u \in V(G) \cup \left\{c\right\}$. If $q$ is a pendent vertex of $G_1$, then $q$ has degree $2$ in $J$ and $n_J(q, u) < n_J(u, q)$ for all vertices $u \in V(G) \cup \left\{c\right\}$. Thus $\peri(v_x) > \peri(v_y)$ for all $x < y$, and $\peri(c) < \peri(v_i)$ for all $i = 1, \dots, n$.

Recall that any clique of size $k+1$ in $J$ cannot contain any pendent vertices of $J$ or any vertices that were pendent vertices in $G_1$, so the only vertices that can be in a clique of size $k+1$ in $J$ are $c$ and the vertices of $G$. Thus $J$ has a clique of size $k+1$ with all vertices having different peripheralities if and only if $G$ has a clique of size $k$. This completes the reduction from $CLIQUE(G,k)$ to $CLIQUE_{peri, \neq}(J,k+1)$.
\end{proof}

We define $CLIQUE_{ecc, \neq}(G,k)$ as the problem of determining whether or not $G$ has a clique on $k$ vertices for which every vertex in the clique has a different eccentricity. Unlike the equality version of this problem, we see that this version is in P.

\begin{thm}\label{eccneq}
The problem $CLIQUE_{ecc, \neq}(G,k)$ is in $P$. 
\end{thm}

\begin{proof}
If $k = 2$, then we can do a brute-force check in polynomial time in $n$ to determine whether $G$ contains a pair of adjacent vertices of different eccentricity. If $k \ge 3$, then it is impossible for $G$ to contain a clique on $k$ vertices for which every vertex in the clique has a different eccentricity. This is because the clique would have to contain two vertices $u, v$ for which $\ecc(u)+2 \le \ecc(v)$. However we must have $\ecc(v) \le \ecc(u)+1$ since $u$ and $v$ are in a clique together (and thus adjacent), so we cannot have $\ecc(u)+2 \le \ecc(v)$.
\end{proof}

We also define $CLIQUE_{eecc, \neq}(G,k)$ as the problem of determining whether or not $G$ has a clique on $k$ vertices for which every edge in the clique has a different edge eccentricity. Unlike the equality version of this problem, and as with $CLIQUE_{ecc, \neq}(G,k)$, we see that this version is in P.

\begin{thm}
The problem $CLIQUE_{eecc, \neq}(G,k)$ is in $P$. 
\end{thm}

\begin{proof}
If $k = 2$, the answer is yes as long as $G$ has an edge. If $k \ge 3$, then it is impossible for $G$ to contain a clique on $k$ vertices for which every edge in the clique has a different edge eccentricity. This is because the clique would have to contain two edges $e, f$ for which $\eecc(e)+2 \le \eecc(f)$. However as in the proof of Theorem \ref{eccneq}, we must have $\eecc(f) \le \eecc(e)+1$ since the distinct vertices in the edges $e$ and $f$ are all adjacent to each other.
\end{proof}

\section{Graphs where the Mostar index is not an accurate measure of peripherality}\label{mostarnot}

It is easy to construct infinite families of graphs where the Mostar index decreases with respect to the distance from the edge to the center. For example, we construct a graph $G_{m,n}$ starting with a copy of the star $K_{1,n}$ with $n > 1$. For each leaf vertex $v$ in $K_{1,n}$, replace it with $m \ge 2$ vertices $v_1, \dots, v_m$ in $G_{m,n}$, and replace the edge between $v$ and the center vertex $c$ of $K_{1, n}$ with $m$ edges, each between $c$ and a vertex $v_i$ with $1 \le i \le m$. For each $1 \le i < j \le m$, we also add an edge between $v_i$ and $v_j$.

The vertex $c$ is clearly the center of $G_{m,n}$ by both degree and eccentricity. However the edges which contain $c$ have positive Mostar index, while the edges that do not contain $c$ have a Mostar index of $0$ by symmetry. Thus in $G_{m,n}$, the Mostar index is higher for edges that are closer to the center. Therefore we have proved the following theorem.

\begin{thm}\label{weirdmostar}
For all $n \ge 5$, there exist radius-$1$ graphs $G$ of order $n$ for which the Mostar index is strictly greater for edges that are closer to the center.
\end{thm}

The last construction can be generalized to produce a much larger family of graphs where the Mostar index decreases with respect to the distance from the edge to the center. Let $G_1, G_2, \dots, G_n$ be any list of $n$ graphs for which $\mo(G_i) = 0$ for each $i = 1, \dots, n$. Let $U(G_1, G_2, \dots, G_n)$ be the graph obtained from the disjoint union of $G_1, G_2, \dots, G_n$ by adding a new universal vertex $u$. Then $u$ is the center of $U(G_1, G_2, \dots, G_n)$ by degree and eccentricity. 

The edges which contain $u$ in $U(G_1, G_2, \dots, G_n)$ have positive Mostar index by construction, but the edges which do not contain $u$ must have Mostar index $0$ since each graph among $G_1, G_2, \dots, G_n$ has Mostar index $0$. Thus in $U(G_1, G_2, \dots, G_n)$, the Mostar index is higher for edges that are closer to the center.

We can also generalize the result in Theorem \ref{weirdmostar} by using a balanced spider with longer legs than a star. We construct a graph $H_{a,b,m}$ starting with a copy of the balanced spider $S_{a, b}$ with $a > m+1$. For each leaf vertex $v$ in $S_{a, b}$, replace it with $m \ge 2$ vertices $v_1, \dots, v_m$ in $G_{m,n}$, and replace the edge between $v$ and its only neighbor $u$ with $m$ edges, each between $u$ and a vertex $v_i$ with $1 \le i \le m$. For each $1 \le i < j \le m$, we also add an edge between $v_i$ and $v_j$.

The center of $H_{a,b,m}$ is the same as the center of $S_{a,b}$ by both degree and eccentricity as long as $a > m+1$. Thus the most peripheral edges in $H_{a,b,m}$ are the edges between two new vertices. By symmetry, any edges between two new vertices in $H_{a,b,m}$ have Mostar index $0$. However, all other edges in $H_{a,b,m}$ have positive Mostar index. Thus in $H_{a,b,m}$, the Mostar index is lowest for the most peripheral edges. Therefore we have proved the following theorem.

\begin{thm}\label{weirdmostarr}
For all $r \ge 1$, we can construct a graph $G$ of radius $r$ for which the Mostar index is lowest for the most peripheral edges.
\end{thm}

We can also find graphs where the vertex with lowest Mostar index is not close to the most central vertex, according to both edge degree and edge eccentricity, even if we restrict ourselves to trees. For example, consider the unbalanced spider obtained from $K_{1,a}$ with center vertex $c$ by adding a leg of length $b > 1$ to $c$. If $a > 1$, then the edge with highest edge degree is the edge on the leg of length $b$ which contains $c$. If $b$ is even and the vertices on the leg of length $b$ are $c = v_0, v_1, v_2, \dots, v_b$ in increasing order of distance to $c$, then the edge with lowest edge eccentricity is $\left\{v_{\frac{b}{2}-1},v_{\frac{b}{2}}\right\}$. If $b$ is odd, then the edges with lowest edge eccentricity are $\left\{v_{\frac{b-1}{2}-1},v_{\frac{b-1}{2}}\right\}$ and $\left\{v_{\frac{b-1}{2}},v_{\frac{b-1}{2}+1}\right\}$.

On the other hand, if $b > a$ and $a+b+1$ is even, then the edge with lowest Mostar index is $\left\{v_{\frac{b-a-1}{2}},v_{\frac{b-a-1}{2}+1}\right\}$. If $b > a$ and $a+b+1$ is odd, then the edges with lowest Mostar index are $\left\{v_{\frac{b-a}{2}-1},v_{\frac{b-a}{2}}\right\}$ and $\left\{v_{\frac{b-a}{2}},v_{\frac{b-a}{2}+1}\right\}$.

If we choose $a \approx \frac{n}{3}$ and $b \approx \frac{2n}{3}$, then we find that the edge(s) with the lowest Mostar index have distance approximately $\frac{n}{6}$ to the two edges that are most central according to edge degree and edge eccentricity. Thus we have the following theorem.

\begin{thm}
There exist trees of order $n$ for which the edges with lowest Mostar index have distance $\Omega(n)$ to the edges which are most central by edge degree and edge eccentricity.
\end{thm}

\section{Centrality and peripherality in SuperFast}\label{superfastperi}

Let $G$ be the graph with vertex set equal to the set of all chemical species that are reactants in the SuperFast reactions, with an edge between two vertices if and only if they are reactants in the same reaction. The SuperFast reactions include a variable $M$ which stands for molecule; we do not include $M$ as a vertex in the graph $G$ since it is a variable rather than a specific chemical species.

More specifically, $G$ has vertices $C H_2 O, C H_3 O_2, C H_3 O O H, C H_4, C O, D M S, H_2 O, H_2 O_2, H O_2,\\$ 
\noindent $I S O P, N O, N O_2, O_3, O H, S O_2$ and edges:

\begin{align*}
\left\{O_3, O H \right\},\left\{H O_2,O_3 \right\},\left\{H O_2, O H \right\}, \left\{H_2 O_2, O H \right\}, \\
\left\{N O, O_3 \right\}, \left\{H O_2, N O \right\},\left\{N O_2, O H \right\},\left\{C H_4, O H \right\}, \\
\left\{C O, O H \right\}, \left\{C H_2 O, O H \right\},\left\{C H_3 O_2, H O_2 \right\}, \left\{C H_3 O O H, O H \right\}, \\
\left\{C H_3 O_2, N O \right\},\left\{H_2 O, N O_2 \right\},\left\{D M S, O H \right\}, \left\{O H, S O_2 \right\}, \\
\left\{H_2 O_2, S O_2 \right\},\left\{O_3, S O_2 \right\},\left\{I S O P, O H\right\}, \left\{I S O P, O_3\right\}.
\end{align*}

Thus $G$ has order $15$, $20$ edges, diameter $4$ and maximum degree $11$. 

\subsection{Centrality and peripherality of vertices in the SuperFast graph}

In \cite{silva}, Silva et al. constructed a directed graph $D$ to represent SuperFast, and determined the reactants with the highest out-degree centrality. In $D$, they included both reactants and products, as well as reactions, as vertices. They found that $O H$ had the highest outdegree centrality in $D$. 

The directed graph from \cite{silva} is different from $G$, since $G$ is undirected and only includes reactants. We investigated centrality and peripherality of the vertices in $G$. By definition, peripherality is the opposite of centrality, so we can turn any peripherality ranking into a centrality ranking by reversal.

In Table \ref{tablecen}, we ranked each chemical species in the graph $G$ with respect to the inverse of peripherality and the inverse of sum peripherality, as well as degree centrality $\deg(v)$, closeness centrality $\cc(v)$, betweenness centrality $\bc(v)$, eigenvector centrality $\ec(v)$, and inverse eccentricity $\ecc(v)^{-1}$.
 
\begin{table}
\begin{tabular}{ | m{1.7cm} | m{1.5cm}| m{1.4cm} | m{1.0cm} |m{1.0cm} |m{1cm} |m{1cm} |m{1.4cm} |} 
  \hline
  reactant $v$ & rank for $\peri(v)^{-1}$ & rank for $\spr(v)^{-1}$ & rank for $\deg(v)$ & rank for $\cc(v)$ & rank for $\bc(v)$ & rank for $\ec(v)$ & rank for $\ecc(v)^{-1}$ \\  
  \hline
CH2O & 8 & 8 & 10 & 8 & 7 & 9 & 2  \\
\hline
CH3O2 & 14 & 14 & 6 & 14 & 7 & 14 & 13  \\
\hline
CH3OOH & 8 & 8 & 10 & 8 & 7 & 9 & 2  \\
\hline
CH4 & 8 & 8 & 10 & 8 & 7 & 9 & 2  \\
\hline
CO & 8 & 8 & 10 & 8 & 7 & 9 & 2  \\
\hline
DMS & 8 & 8 & 10 & 8 & 7 & 9 & 2  \\
\hline
H2O & 15 & 15 & 10 & 15 & 7 & 15 & 13  \\
\hline
H2O2 & 6 & 6 & 6 & 7 & 7 & 6 & 2  \\
\hline
HO2 & 3 & 3 & 3 & 3 & 2 & 3 & 2  \\
\hline
ISOP & 5 & 5 & 6 & 5 & 7 & 5 & 2  \\
\hline
NO & 13 & 13 & 4 & 13 & 5 & 7 & 13  \\
\hline
NO2 & 6 & 7 & 6 & 5 & 3 & 8 & 2  \\
\hline
O3 & 2 & 2 & 2 & 2 & 4 & 2 & 2  \\
\hline
OH & 1 & 1 & 1 & 1 & 1 & 1 & 1  \\
\hline
SO2 & 4 & 4 & 4 & 4 & 6 & 4 & 2  \\
\hline
\end{tabular}
\caption{\label{tablecen} Rank of each chemical species in SuperFast with respect to various centrality measures.}
\end{table}

It is notable that OH is considered the most central vertex by every measure of centrality, and H2O is considered the most peripheral (or tied for it) by all measures. Every measure except betweenness centrality ranks O3 as the second most central. By every measure, HO2 is in the top three most central. Both CH3O2 and NO are in the top three most peripheral (least central) with respect to $\peri(v)^{-1}$, $\spr(v)^{-1}$, $\cc(v)$, and $\ecc(v)^{-1}$. Moreover, there are five chemical species (CH2O, CH3OOH, CH4, CO, DMS) that are tied with each other with respect to every centrality measure in the table.

For the SuperFast graph $G$, degree centrality is accurate at identifying the most central vertices, but $6$ vertices out of $15$ are tied for last with respect to degree centrality, so it is not useful for determining the most peripheral vertices of $G$. This is a similar issue to using degree centrality on a balanced spider. The vertices adjacent to the center have the same degree centrality as the vertices adjacent to the ends of the legs, even though they are much closer to the center for balanced spiders with long legs. In both cases, the degree centrality produces too many ties.

\subsection{Centrality and peripherality of edges in the SuperFast graph}

We also compared each edge in the SuperFast reactant graph with respect to edge degree, inverse edge eccentricity, inverse edge peripherality, inverse edge sum periperality, and inverse Mostar index. The rankings were similar for the first four measures, but the Mostar index usually gave a very different result. This is an example of a graph where the Mostar index does not measure peripherality of edges. Moreover, this is is a natural example, and not an example that has been created artificially for the purpose of having the Mostar index increase with centrality, since the graph is defined from SuperFast.

\begin{table}
\begin{tabular}{ | m{3.2cm} | m{1.5cm} | m{1.5cm}| m{1.5cm} | m{1.5cm} | m{1.5cm} |} 
  \hline
  edge $e$ & rank for $\edeg(e)$ & rank for $\eecc(e)^{-1}$ & rank for $\eperi(e)^{-1}$ & rank for $\espr(e)^{-1}$ & rank for $\mo(e)^{-1}$\\ 
  \hline

CH2O , OH & 4 & 1 & 1 & 8 & 15  \\ 
\hline
CH3O2 , HO2 & 16 & 12 & 16 & 18 & 13  \\ 
\hline
CH3O2 , NO & 18 & 20 & 20 & 20 & 3  \\ 
\hline
CH3OOH , OH & 4 & 1 & 1 & 8 & 15  \\ 
\hline
CH4 , OH & 4 & 1 & 1 & 8 & 15  \\ 
\hline
CO , OH & 4 & 1 & 1 & 8 & 15  \\ 
\hline
DMS , OH & 4 & 1 & 1 & 8 & 15  \\ 
\hline
H2O , NO2 & 20 & 12 & 19 & 19 & 15  \\ 
\hline
H2O2 , OH & 4 & 1 & 1 & 5 & 13  \\ 
\hline
H2O2 , SO2 & 18 & 12 & 18 & 15 & 2  \\ 
\hline
HO2 , NO & 16 & 12 & 16 & 17 & 8  \\ 
\hline
HO2 , O3 & 12 & 12 & 12 & 7 & 1  \\ 
\hline
HO2 , OH & 1 & 1 & 1 & 1 & 7  \\ 
\hline
ISOP , O3 & 15 & 12 & 12 & 14 & 5  \\ 
\hline
ISOP , OH & 4 & 1 & 1 & 4 & 11  \\ 
\hline
NO , O3 & 12 & 12 & 12 & 16 & 9  \\ 
\hline
NO2 , OH & 2 & 1 & 1 & 5 & 11  \\ 
\hline
O3 , OH & 2 & 1 & 1 & 1 & 6  \\ 
\hline
O3 , SO2 & 12 & 12 & 12 & 13 & 3  \\ 
\hline
OH , SO2 & 4 & 1 & 1 & 3 & 9  \\ 
\hline

\end{tabular}
\caption{\label{tablecenedge} Rank of edges in SuperFast with respect to edge degree, inverse edge eccentricity, inverse edge peripherality, inverse edge sum peripherality, and inverse Mostar index.}
\end{table}

Based on the table, the Mostar index appears to be closer to a measure of centrality for the SuperFast network than a measure of peripherality. Overall, the edges with higher Mostar index have higher edge degree and lower edge eccentricity, and the edges with lower Mostar index have lower edge degree and higher edge eccentricity overall. For the edges $\left\{ O_3, O H \right\}$ and $\left\{ H O_2, O H \right\}$, all of the centrality measures except for the Mostar index rank them tied for first, except edge degree ranks $\left\{ O_3, O H \right\}$ second. On the other hand, Mostar index ranks them sixth and seventh respectively. Both of the measures $\eperi(e)$ and $\espr(e)$ appear to be a better measure of peripherality than the Mostar index for the edges in the SuperFast graph, when compared with edge degree and edge eccentricity. 

It is also notable that the edge $\left\{ H_2 O, N O_2\right\}$ is ranked last with respect to edge degree, second to last with respect to inverse edge peripherality and inverse edge sum peripherality, and it is tied for second to last with respect to inverse edge eccentricity. The edge $\left\{C H_3 O_2 , N O \right\}$ was last with respect to inverse edge eccentricity, inverse edge peripherality, and inverse edge sum peripherality.

\section{Centrality and peripherality in MOZART-4}\label{mozart4peri}

As in the last section, we construct a graph $G$ which represents the reactants in a system of atmospheric chemical reactions. In this case, we focus on the system MOZART-4 \cite{emmons}. Let $G$ be the graph with vertex set equal to the set of all chemical species that are reactants in the MOZART-4 reactions, with an edge between two vertices if and only if they are reactants in the same reaction. Thus $G$ has order $81$, $139$ edges, diameter $6$ and maximum degree $54$. As with SuperFast, the MOZART-4 reactions include a variable $M$ which stands for molecule, and we do not include $M$ as a vertex in the graph $G$ since it is a variable rather than a specific chemical species.

\subsection{Centrality and peripherality of vertices in MOZART-4}

In Table \ref{tablecenmozart} in Appendix \ref{mozartvertextable}, we ranked each chemical species in the graph $G$ with respect to the inverse of peripherality and the inverse of sum peripherality, as well as degree centrality $\deg(v)$, closeness centrality $\cc(v)$, betweenness centrality $\bc(v)$, eigenvector centrality $\ec(v)$, and inverse eccentricity $\ecc(v)^{-1}$.
 
It is notable that OH is considered the most central vertex by every measure of centrality except for inverse sum peripherality, which ranks it second to HO2. All of the other measures rank HO2 second most central, except for inverse eccentricity which ranks it third after OH and O. As with SuperFast, H2O is considered the most peripheral by all measures, though it is tied with N2 and N2O for most peripheral on all measures.

The MOZART-4 graph produced some very different rankings across the different measures of centrality. For example: all of CH3O2, CH3CO3, MCO3, NO, NO3, and O1D were in the top ten most central vertices by some measures (including degree centrality), but outside the top fifty by other measures.

\subsection{Centrality and peripherality of edges in MOZART-4}

As with the SuperFast reactant graph, we also compared the edges of the MOZART-4 reactant graph with respect to edge degree, inverse edge eccentricity, inverse edge peripherality, inverse edge sum peripherality, and inverse Mostar index. The results are in Table \ref{tablecenedgemozart} in Appendix \ref{mozartvertextable}. 

With respect to edge degree, it is interesting that the top two most central edges are $\left\{H O_2, O H \right\}$ and $\left\{N O_2, O H \right\}$, in that order. These are the same edges that have the highest edge degrees in the SuperFast reactant graph, though $\left\{O_3, O H\right\}$ is tied with $\left\{N O_2, O H \right\}$ for second in the SuperFast reactant graph. In the MOZART-4 reactant graph, $\left\{O_3, O H\right\}$ is tied for the third highest edge degree with eleven other edges.

In the MOZART-4 reactant graph, any edge containing $O H$ or $O$ is tied for the least eccentricity. In particular, there is a set of $37$ edges which have the same ranking in every category. All of these edges contain $O H$ and a vertex of degree $1$. They are tied for first with respect to Mostar index and edge eccentricity. Note that if $v$ is a pendent vertex that is adjacent to $O H$ in the reactant graph of MOZART-4, then the set of vertices in the reactant graph of MOZART-4 which are closer to $O H$ than to $v$ will consist of all the vertices except for $v$. With respect to edge degree, these $37$ edges tie for fifteenth, and the most peripheral edge with respect to edge degree is $\left\{E O, O_2\right\}$ in the MOZART-4 reactant graph. The next three most peripheral are $\left\{H_2 O, O_1 D \right\}$, $\left\{N_2, O_1 D \right\}$, and $\left\{N_2 O, O_1 D \right\}$. There are $38$ edges tied for the highest edge eccentricity.

The edges $\left\{C H_3 C O_3, C H_3 O_2 \right\}$ and $\left\{N O_2, O \right\}$ were ranked first and second respectively with respect to the inverse Mostar index. For the other centrality measures, $\left\{C H_3 C O_3, C H_3 O_2 \right\}$ ranked outside the top hundred, while $\left\{N O_2, O \right\}$ ranked outside the top hundred for edge degree but tied for first on inverse edge eccentricity. Based on the table, the Mostar index appears to be closer to a measure of centrality for the MOZART-4 network than a measure of peripherality, so it is similar to what we saw with the SuperFast graph.

For the edges $\left\{ O_3, O H \right\}$ and $\left\{ H O_2, O H \right\}$, all of the centrality measures except for the inverse Mostar index rank them in the top three or tied for it. On the other hand, they are not even ranked in the top fifty with respect to the inverse Mostar index. As with the SuperFast graph, both of the measures $\eperi(e)$ and $\espr(e)$ appear to be a better measure of peripherality than the Mostar index for the edges in the MOZART-4 graph, when compared with edge degree and edge eccentricity.

\section{Discussion and future directions}

Based on the results in this paper, we see that the Mostar index is not always an accurate measure of peripherality. There are multiple ways that the Mostar index can disagree with other measures of peripherality. As we saw with the MOZART-4 graph, the edges with highest Mostar index all consisted of the highest-degree vertex $O H$ together with another vertex of degree $1$, but these edges were tied for the most central by edge eccentricity and edge peripherality. 

Another way that the Mostar index can disagree with other measures of peripherality is when there are two vertices $u, v$ whose only neighbors are each other and another vertex $w$. As we saw in Theorems \ref{weirdmostar} and \ref{weirdmostarr}, the edge $\left\{u,v\right\}$ will have Mostar index $0$, even if it is the most peripheral edge by all of the other measures of peripherality and centrality. For another example of an edge with low Mostar index which is not central, recall the MOZART-4 graph, where we saw an edge with the lowest Mostar index for which all of the other measures of peripherality and centrality ranked that edge outside the top $100$ most central.

Sometimes the Mostar index is a fine measure of peripherality, for example with a path or a balanced spider. However, we saw unbalanced spiders where the edge with the lowest Mostar index is not close to the most central edge according to both edge degree and edge eccentricity. 

The Mostar index is still an interesting graph parameter, but it is more a measure of imbalance than peripherality. Another measure of imbalance is irregularity, which measures the imbalance of the degrees of the endpoints in each edge. The Mostar index measures a difference kind of imbalance, the difference between the number of vertices that are closer to each endpoint in the edge.

We proved a number of exact results and extremal results, including refuting a conjecture about the maximum possible Mostar index of bipartite graphs. Although we refuted the conjecture, we did not determine whether our construction attains the maximum, so determining the maximum possible Mostar index of bipartite graphs is still an open problem.

We found that the maximum possible value of $\peri(G)$ among all graphs $G$ of order $n$ is $\binom{n}{2}$ for $n \ge 9$, and we also showed that the maximum possible value of $\spr(G)$ among all graphs $G$ of order $n$ is $\frac{1}{2}n^3-\Theta(n^2)$. For edge peripherality and edge sum peripherality, we showed that the maximum possible value of $\eperi(G)$ among all graphs $G$ of order $n$ is $\Theta(n^3)$ and the maximum possible value of $\espr(G)$ among all graphs $G$ of order $n$ is $\Theta(n^4)$. It is an open problem to improve these bounds.

We also showed that the maximum possible difference between the Mostar index and the irregularity of a tree of order $n$ is $n^2(1-o(1))$. More specifically, we showed that the maximum difference is $n^2-O(\frac{\log{n}}{\log{\log{n}}}n)$, so another problem is to improve these bounds or determine the exact value of the maximum difference.

An interesting direction for future research is the Mostar index of random graphs. For example, what is the expected Mostar index of $G_{n,p}$? Another problem is to determine the expected values for $\peri(G_{n,p})$, $\spr(G_{n,p})$, $\eperi(G_{n,p})$, and $\espr(G_{n,p})$. We found one result in this direction.

\begin{thm}
The expected value of $\irr(G_{n,\frac{1}{2}})$ is $\frac{n^{5/2}(1-o(1))}{4\sqrt{\pi}}$.
\end{thm}

\begin{proof}
Let $u, v$ be any two distinct vertices in $G_{n, \frac{1}{2}}$. In order to compute $|d_u-d_v|$, consider independent binomial random variables $X$ and $Y$ with $m = n-1$ fair coin flips. The probability that $X = Y$ is
$$ \frac{1}{2^{2m}}\sum_{i = 0}^m \binom{m}{i}^2 = \frac{\binom{2m}{m}}{2^{2m}}.$$

For each $k > 0$, the probability that $|X - Y| = k$ is
$$ \frac{2}{2^{2m}}\sum_{i = 0}^{m-k} \binom{m}{i}\binom{m}{i+k} = \frac{\binom{2m}{m+k}}{2^{2m-1}}.$$

Thus the expected value of $|X-Y|$ is 
$$ \sum_{k = 1}^m k \frac{\binom{2m}{m+k}}{2^{2m-1}} = \frac{(m+1)\binom{2m}{m+1}}{2^{2m}} = \sqrt{\frac{n}{\pi}}(1-o(1)).$$

In the last line, we used the result of Graham et al. \cite{graham} that $\sum_{k \le m} \binom{r}{k}(\frac{r}{2}-k) = \frac{m+1}{2}\binom{r}{m+1}$ for integers $m$.

Since there are $\binom{n}{2}$ unordered pairs of distinct vertices $u,v$, and each has probability $\frac{1}{2}$ of being an edge in $G_{n,\frac{1}{2}}$, the expected value of $\irr(G_{n,\frac{1}{2}})$ is $\frac{n^{5/2}(1-o(1))}{4\sqrt{\pi}}$ by linearity of expectation.
\end{proof}

Another interesting problem is to investigate the differences between the rankings of the edges of $G_{n,p}$ according to various measures of peripherality. How do the rankings of the edges according to the Mostar index compare with other measures of peripherality on $G_{n,p}$?

There are also a number of interesting computational problems concerning the different measures of centrality and peripherality in this paper. In Section \ref{npcliqueproblems}, we classified several problems as either in P or NP-complete. One problem that we did not classify is $CLIQUE_{eperi, \neq}(G,k)$, which is the problem of determining whether or not $G$ has a clique on $k$ vertices for which every edge in the clique has a different edge peripherality. We conjecture that this problem is NP-complete. Similar problems can also be investigated for $\spr(v)$ and $\espr(e)$, with either equality constraints or inequality constraints.

\section*{Acknowledgement} JG was supported by the Woodward Fund for Applied Mathematics at San Jose State University, a gift from the estate of Mrs. Marie Woodward in memory of her son, Henry Tynham Woodward. He was an alumnus of the Mathematics Department at San Jose State University and worked with research groups at NASA Ames.

\appendices

\newpage 
\section{MOZART-4 tables}\label{mozartvertextable}

\begin{table}[htb]
\small
\caption{\label{tablecenmozart} Rank of each chemical species in MOZART-4 with respect to various centrality measures.}
\begin{tabular}{ | m{2.4cm} | m{1.5cm}| m{1.4cm} | m{1.0cm} |m{1.0cm} |m{1cm} |m{1cm} |m{1.4cm} |} 
  \hline
  reactant $v$ & rank for $\peri(v)^{-1}$ & rank for $\spr(v)^{-1}$ & rank for $\deg(v)$ & rank for $\cc(v)$ & rank for $\bc(v)$ & rank for $\ec(v)$ & rank for $\ecc(v)^{-1}$ \\ 
  \hline
ALKO2 & 68 & 68 & 23 & 68 & 26 & 21 & 57  \\
\hline
ALKOOH & 19 & 19 & 38 & 19 & 26 & 36 & 3  \\
\hline
BIGALK & 19 & 19 & 38 & 19 & 26 & 36 & 3  \\
\hline
BIGENE & 19 & 19 & 38 & 19 & 26 & 36 & 3  \\
\hline
C10H16 & 6 & 6 & 15 & 6 & 14 & 18 & 3  \\
\hline
C2H4 & 16 & 14 & 23 & 16 & 26 & 31 & 3  \\
\hline
C2H5O2 & 65 & 65 & 15 & 65 & 26 & 15 & 57  \\
\hline
C2H5OH & 19 & 19 & 38 & 19 & 26 & 36 & 3  \\
\hline
C2H5OOH & 19 & 19 & 38 & 19 & 26 & 36 & 3  \\
\hline
C2H6 & 19 & 19 & 38 & 19 & 26 & 36 & 3  \\
\hline
C3H6 & 6 & 6 & 15 & 6 & 14 & 18 & 3  \\
\hline
C3H7O2 & 65 & 65 & 15 & 65 & 26 & 15 & 57  \\
\hline
C3H7OOH & 19 & 19 & 38 & 19 & 26 & 36 & 3  \\
\hline
C3H8 & 19 & 19 & 38 & 19 & 26 & 36 & 3  \\
\hline
CH2O & 9 & 9 & 23 & 9 & 17 & 26 & 3  \\
\hline
CH3CHO & 9 & 9 & 23 & 9 & 17 & 26 & 3  \\
\hline
CH3CO3 & 64 & 63 & 7 & 64 & 13 & 8 & 57  \\
\hline
CH3COCH3 & 19 & 19 & 38 & 19 & 26 & 36 & 3  \\
\hline
CH3COCHO & 9 & 9 & 23 & 9 & 17 & 26 & 3  \\
\hline
CH3COOH & 19 & 19 & 38 & 19 & 26 & 36 & 3  \\
\hline
CH3COOOH & 19 & 19 & 38 & 19 & 26 & 36 & 3  \\
\hline
CH3O2 & 63 & 64 & 6 & 63 & 12 & 6 & 57  \\
\hline
CH3OH & 19 & 19 & 38 & 19 & 26 & 36 & 3  \\
\hline
CH3OOH & 19 & 19 & 38 & 19 & 26 & 36 & 3  \\
\hline
CH4 & 14 & 17 & 23 & 9 & 9 & 34 & 3  \\
\hline
CO & 19 & 19 & 38 & 19 & 26 & 36 & 3  \\
\hline
CRESOL & 19 & 19 & 38 & 19 & 26 & 36 & 3  \\
\hline
DMS & 9 & 9 & 23 & 9 & 17 & 26 & 3  \\
\hline
ENEO2 & 76 & 75 & 38 & 76 & 26 & 73 & 77  \\
\hline
EO & 78 & 81 & 38 & 78 & 26 & 78 & 57  \\
\hline
  EO2 & 76 & 75 & 38 & 76 & 26 & 73 & 77  \\
\hline
GLYALD & 19 & 19 & 38 & 19 & 26 & 36 & 3  \\
\hline
GLYOXAL & 19 & 19 & 38 & 19 & 26 & 36 & 3  \\
\hline
H2 & 14 & 17 & 23 & 9 & 9 & 34 & 3  \\
\hline

\end{tabular}
\end{table}

\begin{table}[htb]
\begin{tabular}{ | m{2.4cm} | m{1.5cm}| m{1.4cm} | m{1.0cm} |m{1.0cm} |m{1cm} |m{1cm} |m{1.4cm} |} 
  \hline

H2O & 79 & 78 & 38 & 79 & 26 & 79 & 77  \\
\hline
H2O2 & 19 & 19 & 38 & 19 & 26 & 36 & 3  \\
\hline  
HNO3 & 19 & 19 & 38 & 19 & 26 & 36 & 3  \\
\hline
HO2 & 2 & 1 & 2 & 2 & 2 & 2 & 3  \\
\hline
HO2NO2 & 19 & 19 & 38 & 19 & 26 & 36 & 3  \\
\hline
HYAC & 19 & 19 & 38 & 19 & 26 & 36 & 3  \\
\hline
HYDRALD & 19 & 19 & 38 & 19 & 26 & 36 & 3  \\
\hline
ISOP & 6 & 6 & 15 & 6 & 14 & 18 & 3  \\
\hline
ISOPNO3 & 62 & 62 & 15 & 62 & 26 & 14 & 57  \\
\hline
ISOPO2 & 59 & 59 & 11 & 59 & 23 & 10 & 57  \\
\hline
ISOPOOH & 19 & 19 & 38 & 19 & 26 & 36 & 3  \\
\hline
MACR & 16 & 14 & 23 & 16 & 26 & 31 & 3  \\
\hline
MACRO2 & 59 & 59 & 11 & 59 & 23 & 10 & 57  \\
\hline
MACROOH & 19 & 19 & 38 & 19 & 26 & 36 & 3  \\
\hline
MCO3 & 58 & 58 & 9 & 58 & 22 & 9 & 57  \\
\hline
MEK & 19 & 19 & 38 & 19 & 26 & 36 & 3  \\
\hline
MEKO2 & 68 & 68 & 23 & 68 & 26 & 21 & 57  \\
\hline
MEKOOH & 19 & 19 & 38 & 19 & 26 & 36 & 3  \\
\hline
MPAN & 19 & 19 & 38 & 19 & 26 & 36 & 3  \\
\hline
MVK & 16 & 14 & 23 & 16 & 26 & 31 & 3  \\
\hline
N2 & 79 & 78 & 38 & 79 & 26 & 79 & 77  \\
\hline
N2O & 79 & 78 & 38 & 79 & 26 & 79 & 77  \\
\hline
NH3 & 19 & 19 & 38 & 19 & 26 & 36 & 3  \\
\hline
NO & 56 & 56 & 3 & 56 & 3 & 3 & 57  \\
\hline
NO2 & 5 & 4 & 7 & 5 & 8 & 7 & 3  \\
\hline
NO3 & 57 & 57 & 4 & 57 & 7 & 4 & 57  \\
\hline
O & 4 & 5 & 11 & 4 & 5 & 13 & 1  \\
\hline
O1D & 75 & 77 & 9 & 75 & 4 & 77 & 57  \\
\hline
O2 & 73 & 73 & 15 & 73 & 11 & 76 & 3  \\
\hline
O3 & 3 & 3 & 5 & 3 & 6 & 5 & 3  \\
\hline
OH & 1 & 2 & 1 & 1 & 1 & 1 & 1  \\
\hline
ONIT & 19 & 19 & 38 & 19 & 26 & 36 & 3  \\
\hline
ONITR & 9 & 9 & 23 & 9 & 17 & 26 & 3  \\
\hline
PAN & 19 & 19 & 38 & 19 & 26 & 36 & 3  \\
\hline
PO2 & 68 & 68 & 23 & 68 & 26 & 21 & 57  \\
\hline
POOH & 19 & 19 & 38 & 19 & 26 & 36 & 3  \\
\hline
RO2 & 65 & 65 & 15 & 65 & 26 & 15 & 57  \\
\hline
ROOH & 19 & 19 & 38 & 19 & 26 & 36 & 3  \\
\hline
SO2 & 19 & 19 & 38 & 19 & 26 & 36 & 3  \\
\hline
TERPO2 & 68 & 68 & 23 & 68 & 26 & 21 & 57  \\
\hline
TERPOOH & 19 & 19 & 38 & 19 & 26 & 36 & 3  \\
\hline
TOLO2 & 68 & 68 & 23 & 68 & 26 & 21 & 57  \\
\hline

TOLOOH & 19 & 19 & 38 & 19 & 26 & 36 & 3  \\
\hline
TOLUENE & 19 & 19 & 38 & 19 & 26 & 36 & 3  \\
\hline
XO2 & 59 & 59 & 11 & 59 & 23 & 10 & 57  \\
\hline
XOH & 74 & 74 & 38 & 74 & 26 & 75 & 57  \\
\hline
XOOH & 19 & 19 & 38 & 19 & 26 & 36 & 3  \\
\hline
\end{tabular}
\end{table}

\newpage 

\begin{table}[htb]
\small
\caption{\label{tablecenedgemozart} Rank of edges in MOZART-4 with respect to edge degree, inverse edge eccentricity, inverse edge peripherality, inverse edge sum peripherality, and inverse Mostar index.}
\begin{tabular}{ | m{3.5cm} | m{1.5cm} | m{1.5cm}| m{1.5cm} | m{1.5cm} | m{1.5cm} |} 
  \hline
  edge $e$ & rank for $\edeg(e)$ & rank for $\eecc(e)^{-1}$ & rank for $\eperi(e)^{-1}$ & rank for $\espr(e)^{-1}$ & rank for $\mo(e)^{-1}$\\ 
  \hline
ALKO2 , HO2 & 62 & 59 & 55 & 90 & 91  \\ 
\hline
ALKO2 , NO & 80 & 102 & 100 & 116 & 47  \\ 
\hline
ALKOOH , OH & 15 & 1 & 1 & 27 & 96  \\ 
\hline
BIGALK , OH & 15 & 1 & 1 & 27 & 96  \\ 
\hline
BIGENE , OH & 15 & 1 & 1 & 27 & 96  \\ 
\hline
C10H16 , NO3 & 95 & 59 & 90 & 77 & 34  \\ 
\hline
C10H16 , O3 & 110 & 59 & 75 & 19 & 41  \\ 
\hline
C10H16 , OH & 3 & 1 & 1 & 11 & 73  \\ 
\hline
C2H4 , O3 & 115 & 59 & 75 & 27 & 52  \\ 
\hline
C2H4 , OH & 15 & 1 & 1 & 22 & 85  \\ 
\hline
C2H5O2 , CH3O2 & 122 & 102 & 131 & 130 & 14  \\ 
\hline
C2H5O2 , HO2 & 62 & 59 & 55 & 87 & 85  \\ 
\hline
C2H5O2 , NO & 80 & 102 & 100 & 113 & 44  \\ 
\hline
C2H5OH , OH & 15 & 1 & 1 & 27 & 96  \\ 
\hline
C2H5OOH , OH & 15 & 1 & 1 & 27 & 96  \\ 
\hline
C2H6 , OH & 15 & 1 & 1 & 27 & 96  \\ 
\hline
C3H6 , NO3 & 95 & 59 & 90 & 77 & 34  \\ 
\hline
C3H6 , O3 & 110 & 59 & 75 & 19 & 41  \\ 
\hline
C3H6 , OH & 3 & 1 & 1 & 11 & 73  \\ 
\hline
C3H7O2 , CH3O2 & 122 & 102 & 131 & 130 & 14  \\ 
\hline
C3H7O2 , HO2 & 62 & 59 & 55 & 87 & 85  \\ 
\hline
C3H7O2 , NO & 80 & 102 & 100 & 113 & 44  \\ 
\hline
C3H7OOH , OH & 15 & 1 & 1 & 27 & 96  \\ 
\hline
C3H8 , OH & 15 & 1 & 1 & 27 & 96  \\ 
\hline
CH2O , NO3 & 102 & 59 & 93 & 80 & 29  \\ 
\hline
CH2O , OH & 3 & 1 & 1 & 14 & 77  \\ 
\hline
CH3CHO , NO3 & 102 & 59 & 93 & 80 & 29  \\ 
\hline

  CH3CHO , OH & 3 & 1 & 1 & 14 & 77  \\ 
\hline
CH3CO3 , CH3O2 & 115 & 102 & 131 & 129 & 1  \\ 
\hline
CH3CO3 , HO2 & 62 & 59 & 55 & 85 & 77  \\ 
\hline
CH3CO3 , ISOPO2 & 126 & 102 & 125 & 123 & 7  \\ 
\hline
  CH3CO3 , MACRO2 & 126 & 102 & 125 & 123 & 7  \\ 
\hline
CH3CO3 , MCO3 & 126 & 102 & 123 & 121 & 12  \\ 
\hline
CH3CO3 , NO & 62 & 102 & 100 & 111 & 40  \\ 
\hline
CH3CO3 , NO2 & 109 & 59 & 86 & 95 & 64  \\ 
\hline
  CH3CO3 , XO2 & 126 & 102 & 125 & 123 & 7  \\ 
\hline
CH3COCH3 , OH & 15 & 1 & 1 & 27 & 96  \\ 
\hline
CH3COCHO , NO3 & 102 & 59 & 93 & 80 & 29  \\ 
\hline
CH3COCHO , OH & 3 & 1 & 1 & 14 & 77  \\ 
\hline
CH3COOH , OH & 15 & 1 & 1 & 27 & 96  \\ 
\hline
CH3COOOH , OH & 15 & 1 & 1 & 27 & 96  \\ 
\hline
CH3O2 , HO2 & 62 & 59 & 55 & 86 & 73  \\ 
\hline
CH3O2 , ISOPO2 & 115 & 102 & 125 & 126 & 3  \\ 
\hline
CH3O2 , MACRO2 & 115 & 102 & 125 & 126 & 3  \\ 
\hline
CH3O2 , MCO3 & 110 & 102 & 123 & 122 & 10  \\ 
\hline
CH3O2 , NO & 80 & 102 & 100 & 112 & 39  \\ 
\hline

\end{tabular}
\end{table}

\begin{table}
\begin{tabular}{| m{3.5cm} | m{1.5cm}| m{1.5cm} |m{1.5cm} | m{1.5cm} | m{1.5cm} |} 
  \hline

CH3O2 , RO2 & 122 & 102 & 131 & 130 & 14  \\ 
\hline
CH3O2 , XO2 & 115 & 102 & 125 & 126 & 3  \\ 
\hline
CH3OH , OH & 15 & 1 & 1 & 27 & 96  \\ 
\hline
CH3OOH , OH & 15 & 1 & 1 & 27 & 96  \\ 
\hline
CH4 , O1D & 133 & 59 & 98 & 98 & 69  \\ 
\hline
CH4 , OH & 3 & 1 & 1 & 22 & 77  \\ 
\hline
CO , OH & 15 & 1 & 1 & 27 & 96  \\ 
\hline
CRESOL , OH & 15 & 1 & 1 & 27 & 96  \\ 
\hline
DMS , NO3 & 102 & 59 & 93 & 80 & 29  \\ 
\hline
DMS , OH & 3 & 1 & 1 & 14 & 77  \\ 
\hline
ENEO2 , NO & 80 & 102 & 100 & 133 & 96  \\ 
\hline
EO , O2 & 139 & 59 & 135 & 136 & 96  \\ 
\hline
EO2 , NO & 80 & 102 & 100 & 133 & 96  \\ 
\hline
GLYALD , OH & 15 & 1 & 1 & 27 & 96  \\ 
\hline
GLYOXAL , OH & 15 & 1 & 1 & 27 & 96  \\ 
\hline
H2 , O1D & 133 & 59 & 98 & 98 & 69  \\ 
\hline
H2 , OH & 3 & 1 & 1 & 22 & 77  \\ 
\hline
H2O , O1D & 136 & 102 & 137 & 137 & 96  \\ 
\hline
H2O2 , OH & 15 & 1 & 1 & 27 & 96  \\ 
\hline
HNO3 , OH & 15 & 1 & 1 & 27 & 96  \\ 
\hline
HO2 , ISOPNO3 & 62 & 59 & 55 & 74 & 71  \\ 
\hline
HO2 , ISOPO2 & 62 & 59 & 55 & 71 & 66  \\ 
\hline
HO2 , MACRO2 & 62 & 59 & 55 & 71 & 66  \\ 
\hline
HO2 , MCO3 & 62 & 59 & 55 & 70 & 65  \\ 
\hline
HO2 , MEKO2 & 62 & 59 & 55 & 90 & 91  \\ 
\hline
HO2 , NO & 59 & 59 & 55 & 67 & 59  \\ 
\hline
HO2 , NO2 & 60 & 59 & 55 & 4 & 24  \\ 
\hline
HO2 , NO3 & 55 & 59 & 55 & 69 & 63  \\ 
\hline
HO2 , O & 60 & 1 & 55 & 5 & 22  \\ 
\hline
HO2 , O3 & 58 & 59 & 55 & 2 & 17  \\ 
\hline
HO2 , OH & 1 & 1 & 1 & 1 & 55  \\ 
\hline
HO2 , PO2 & 62 & 59 & 55 & 90 & 91  \\ 
\hline
HO2 , RO2 & 62 & 59 & 55 & 87 & 85  \\ 
\hline
HO2 , TERPO2 & 62 & 59 & 55 & 90 & 91  \\ 
\hline
HO2 , TOLO2 & 62 & 59 & 55 & 90 & 91  \\ 
\hline
HO2 , XO2 & 62 & 59 & 55 & 71 & 66  \\ 
\hline
HO2NO2 , OH & 15 & 1 & 1 & 27 & 96  \\ 
\hline
HYAC , OH & 15 & 1 & 1 & 27 & 96  \\ 
\hline
HYDRALD , OH & 15 & 1 & 1 & 27 & 96  \\ 
\hline
ISOP , NO3 & 95 & 59 & 90 & 77 & 34  \\ 
\hline
ISOP , O3 & 110 & 59 & 75 & 19 & 41  \\ 
\hline
ISOP , OH & 3 & 1 & 1 & 11 & 73  \\ 
\hline
ISOPNO3 , NO & 80 & 102 & 100 & 105 & 34  \\ 
\hline
ISOPNO3 , NO3 & 107 & 102 & 118 & 110 & 23  \\ 
\hline
ISOPO2 , NO & 80 & 102 & 100 & 102 & 26  \\ 
\hline
ISOPO2 , NO3 & 95 & 102 & 118 & 107 & 19  \\ 
\hline

\end{tabular}
\end{table}

\begin{table}
\begin{tabular}{| m{3.5cm} | m{1.5cm}| m{1.5cm} |m{1.5cm} | m{1.5cm} | m{1.5cm} |} 
  \hline

ISOPOOH , OH & 15 & 1 & 1 & 27 & 96  \\ 
\hline
MACR , O3 & 115 & 59 & 75 & 27 & 52  \\ 
\hline
MACR , OH & 15 & 1 & 1 & 22 & 85  \\ 
\hline
MACRO2 , NO & 80 & 102 & 100 & 102 & 26  \\ 
\hline
MACRO2 , NO3 & 95 & 102 & 118 & 107 & 19  \\ 
\hline
MACROOH , OH & 15 & 1 & 1 & 27 & 96  \\ 
\hline
MCO3 , NO & 62 & 102 & 100 & 101 & 25  \\ 
\hline
MCO3 , NO2 & 122 & 59 & 86 & 76 & 60  \\ 
\hline
MCO3 , NO3 & 95 & 102 & 118 & 106 & 17  \\ 
\hline
MEK , OH & 15 & 1 & 1 & 27 & 96  \\ 
\hline
MEKO2 , NO & 80 & 102 & 100 & 116 & 47  \\ 
\hline
MEKOOH , OH & 15 & 1 & 1 & 27 & 96  \\ 
\hline
MPAN , OH & 15 & 1 & 1 & 27 & 96  \\ 
\hline
MVK , O3 & 115 & 59 & 75 & 27 & 52  \\ 
\hline
MVK , OH & 15 & 1 & 1 & 22 & 85  \\ 
\hline
N2 , O1D & 136 & 102 & 137 & 137 & 96  \\ 
\hline
N2O , O1D & 136 & 102 & 137 & 137 & 96  \\ 
\hline
NH3 , OH & 15 & 1 & 1 & 27 & 96  \\ 
\hline
NO , NO3 & 55 & 102 & 100 & 100 & 12  \\ 
\hline
NO , O3 & 55 & 59 & 75 & 68 & 56  \\ 
\hline
NO , PO2 & 80 & 102 & 100 & 116 & 47  \\ 
\hline
NO , RO2 & 80 & 102 & 100 & 113 & 44  \\ 
\hline
NO , TERPO2 & 80 & 102 & 100 & 116 & 47  \\ 
\hline
NO , TOLO2 & 80 & 102 & 100 & 116 & 47  \\ 
\hline
NO , XO2 & 80 & 102 & 100 & 102 & 26  \\ 
\hline
NO2 , NO3 & 62 & 59 & 86 & 75 & 57  \\ 
\hline
NO2 , O & 126 & 1 & 84 & 10 & 2  \\ 
\hline
NO2 , O3 & 108 & 59 & 75 & 8 & 10  \\ 
\hline
NO2 , OH & 2 & 1 & 1 & 6 & 62  \\ 
\hline
NO2 , XOH & 131 & 59 & 86 & 97 & 96  \\ 
\hline
NO3 , ONITR & 102 & 59 & 93 & 80 & 29  \\ 
\hline
NO3 , XO2 & 95 & 102 & 118 & 107 & 19  \\ 
\hline
O , O2 & 133 & 1 & 84 & 96 & 71  \\ 
\hline
O , O3 & 110 & 1 & 75 & 9 & 3  \\ 
\hline
O , OH & 3 & 1 & 1 & 7 & 60  \\ 
\hline
O1D , O2 & 131 & 59 & 135 & 135 & 34  \\ 
\hline
O3 , OH & 3 & 1 & 1 & 3 & 58  \\ 
\hline
OH , ONIT & 15 & 1 & 1 & 27 & 96  \\ 
\hline
OH , ONITR & 3 & 1 & 1 & 14 & 77  \\ 
\hline
OH , PAN & 15 & 1 & 1 & 27 & 96  \\ 
\hline
OH , POOH & 15 & 1 & 1 & 27 & 96  \\ 
\hline
OH , ROOH & 15 & 1 & 1 & 27 & 96  \\ 
\hline
OH , SO2 & 15 & 1 & 1 & 27 & 96  \\ 
\hline
OH , TERPOOH & 15 & 1 & 1 & 27 & 96  \\ 
\hline
OH , TOLOOH & 15 & 1 & 1 & 27 & 96  \\ 
\hline
OH , TOLUENE & 15 & 1 & 1 & 27 & 96  \\ 
\hline
OH , XOOH & 15 & 1 & 1 & 27 & 96  \\ 
\hline

\end{tabular}

\end{table}

\end{document}